\documentclass[a4paper]{amsart}                

\usepackage[latin1]{inputenc}                  
\usepackage[T1]{fontenc}                       
\usepackage[spanish,english]{babel}            
\usepackage{graphicx}                  
\usepackage{amsmath,amssymb,amsthm}            
\usepackage{latexsym}                          
\usepackage{delarray}                          
\usepackage{bbm}                               
\usepackage{hyperref}                          
\usepackage{calrsfs,eufrak,mathrsfs,upgreek}   
\usepackage{bbding,trfsigns}                   
\usepackage{wasysym}                           
\usepackage{datetime}

\DeclareMathAlphabet{\mathpzc}{OT1}{pzc}{m}{it} 

\newtheorem{teo}{Theorem}

\newtheorem{lema}{Lemma}
\newtheorem{propo}{Proposition}

\author[J. Betancor]{J.J. Betancor}

\author[A.J. Castro]{A.J. Castro}

\author[J. Curbelo]{J. Curbelo}

\author[J.C. Fari\~{n}a]{J.C. Fari\~{n}a}

\author[L. Rodr\'{\i}guez-Mesa]{L. Rodr\'{\i}guez-Mesa}
\address{\newline
        Jorge J. Betancor, Alejandro J. Castro, Juan C. Fari\~na and Lourdes Rodr\'{\i}guez-Mesa \newline
        Departamento de An\'alisis Matem\'atico,
        Universidad de la Laguna, \newline
        Campus de Anchieta, Avda. Astrof\'{\i}sico Francisco S\'anchez, s/n, \newline
        38271, La Laguna (Sta. Cruz de Tenerife), Spain}
\email{jbetanco@ull.es, ajcastro@ull.es, jcfarina@ull.es, lrguez@ull.es\vspace*{-0.3cm}}

\address{\newline
        Jezabel Curbelo \newline
        Instituto de Ciencias Matem\'aticas (CSIC-UAM-UCM-UC3M), \newline
        Nicol\'as Cabrera, 13-15, \newline
        28049, Madrid, Spain}
\email{jezabel.curbelo@icmat.es}

\thanks{This paper is partially supported by MTM2010/17974 . The second author is also supported by a FPU grant from the Government of Spain.
        The third author is supported by a grant JAE-Predoc of the CSIC (Spain).}
\date{\today}

\begin{document}

\title[Hermite square functions for vector-valued functions in UMD spaces]
{Square functions in the Hermite setting for functions with values in UMD spaces}

\begin{abstract}
In this paper we characterize the Lebesgue Bochner spaces $L^p(\mathbb{R}^n,B)$, $1<p<\infty$, by using Littlewood-Paley $g$-functions in the Hermite setting, provided that $B$ is a UMD Banach space. We use $\gamma$-radonifying operators $\gamma (H,B)$ where $H=L^2((0,\infty ),\frac{dt}{t})$. We also characterize the UMD Banach spaces in terms of $L^p(\mathbb{R}^n,B)-L^p(\mathbb{R}^n,\gamma (H,B))$ boundedness of Hermite Littlewood-Paley $g$-functions.
\end{abstract}

\maketitle

\section{Introduction}
As it is wellknown the Hilbert transform $\mathcal{H}(f)$ of $f\in L^p(\mathbb R)$, $1 \leq p < \infty$, is defined by
$$
\mathcal{H}(f)(x)= \frac{1}{\pi}\lim_{\varepsilon \rightarrow 0^+}\int_{|x-y|>\varepsilon}\frac{f(y)}{x-y}dy,\;\;\mbox{a.e.}\;\;x \in\mathbb R.
$$
The operator $\mathcal{H}$ is bounded from $L^p(\mathbb R)$ into itself, for every $1<p<\infty$, and from $L^1(\mathbb R)$ into $L^{1,\infty}(\mathbb R)$. If $1\leq p < \infty$ and $B$ is a Banach space, the Hilbert transform is defined on $L^p(\mathbb R)\otimes B$ in a natural way. A Banach space $B$ is said to be a UMD space when the Hilbert transform $\mathcal{H}$ can be extended to the Bochnner-Lebesgue space $L^p(\mathbb R, B)$ as a bounded operator from $L^p(\mathbb R, B)$ into itself, for some $1<p<\infty$. The UMD property does not depend on $p$. Indeed, if $\mathcal{H}$ can be extended to $L^p(\mathbb R,B)$ as a bounded operator from  $L^p(\mathbb R, B)$ into itself, for some $1<p<\infty$, then this property holds for every $1<p<\infty$. Moreover, $B$ is a UMD Banach space if and only if the Hilbert transform can be extended to $L^1(\mathbb R,B)$ as a bounded operator from $L^1(\mathbb R, B)$ into $L^{1,\infty}(\mathbb R,B)$. There exist many other characterizations for the UMD Banach spaces (\cite{AT}, \cite{BFRST}, \cite{BFMT}, \cite{Bo1}, \cite{Bu2}, \cite{HTV1} and \cite{Hy2}, amongst others).

Inspired by the results due to Kaiser and Weis \cite{KW}, in this paper we characterize the Bochner-Lebesgue spaces $L^p(\mathbb R^n,B)$, $1<p<\infty$, where $B$ is a UMD space, by using Littlewood-Paley $g$-functions associated with the Poisson semigroup for the Hermite operator and $\gamma$-radonifying operators. We also obtain new characterizations of UMD Banach spaces.

We consider the Hermite (also called harmonic oscillator) operator
$$
L=-\Delta + |x|^2,\;\; \mbox{on}\;\; \mathbb R^n,\;n \geq 1,
$$
where $\Delta$ denotes the Euclidean Laplacian. For every $k \in \mathbb N$, we define the $k$-th Hermite function by
$$
h_k(z) = (\sqrt \pi 2^k k!)^{-\frac{1}{2}}H_k(z) e^{-\frac{z^2}{2}}, \;\;z \in \mathbb R,
$$
being $H_k$ the $k$-th Hermite polynomial (\cite[p. 106]{Sz}) given by
$H_k(z)=(-1)^ke^{z^2}\frac{d^k}{dz^k}e^{-z^2}$, $z\in \mathbb{R}$.

If $k=(k_1,\ldots,k_n)\in \mathbb N^n$, and $h_k(x)= \prod^n_{j=1}h_{k_j}(x_j)$,  $x=(x_1, \ldots,x_n) \in \mathbb R^n$, we have that
$$
Lh_k= (2|k|+n)h_k,
$$
where $|k|= k_1+\ldots+k_n$. The system $\{h_k\}_{k\in \mathbb N^n}$ is orthonormal and complete in $L^2(\mathbb R^n)$. The operator $\mathcal{L}$ is defined by
$$
\mathcal{L}(f)= \sum_{k\in \mathbb N^n}(2|k|+n)\langle f,h_k\rangle h_k,\;\; f \in D(\mathcal{L}),
$$
where
$$
D(\mathcal{L})= \{f \in L^2(\mathbb R^n): \sum_{k\in \mathbb N^n}|(2|k|+n)\langle f,h_k\rangle|^2 < \infty\}.
$$
Here, for every $k \in \mathbb N^n$ and $f \in L^2(\mathbb R^n)$, $\langle f,h_k\rangle = \int_{\mathbb R^n}f(y)h_k(y)dy$. It is clear that $\mathcal{L}(f)= L(f)$, when $f$ belongs to $C_c^\infty(\mathbb R^n)$, the space of smooth functions with compact support in $\mathbb R^n$. The operator $-\mathcal{L}$ generates on $L^2(\mathbb R^n)$ the semigroup of operators $\{W_t^{\mathcal{L}}\}_{t>0}$ where, for every $t>0$,
$$
W_t^{\mathcal{L}}(f) = \sum_{k\in \mathbb N^n}e^{-(2|k|+n)t}\langle f,h_k\rangle h_k,\;\; f \in L^2(\mathbb R^n).
$$
According to the Mehler's formula (\cite[(1.1.36)]{Th2}) we get, for every $f  \in L^2(\mathbb R^n)$,
\begin{equation}\label{Hersemi}
W_t^{\mathcal{L}}(f)(x)=\int_{\mathbb R^n}W_t^{\mathcal{L}}(x,y)f(y)dy,\quad x\in \mathbb R^n \mbox{ and } t>0,
\end{equation}
being
$$
W_t^{\mathcal{L}}(x,y) = \frac{1}{\pi^{\frac{n}2}}\left(\frac{e^{-2t}}{1-e^{-4t}}\right)^{\frac{n}{2}}\exp\left(-\frac{|x|^2+|y|^2}{2}\frac{1+e^{-4t}}{1-e^{-4t}}+\frac{2e^{-2t}x\cdot y}{1-e^{-4t}}\right), \;\;x,y \in \mathbb R^n\mbox{ and } t>0.
$$
By defining $W_t^{\mathcal{L}}(f)$, for every $f \in L^p(\mathbb R^n)$, by (\ref{Hersemi}), the family $\{W_t^{\mathcal{L}}\}_{t>0}$ is a positive semigroup of contractions in $L^p(\mathbb R^n)$, for every $1 \leq p \leq \infty$.

The Poisson semigroup associated to $\mathcal{L}$ (generated by $-\sqrt \mathcal{L}$) on $L^p(\mathbb R^n)$ can be written by using the subordination formula as follows
$$
P_t^\mathcal{L}(f) = \frac{t}{\sqrt{4\pi}}\int_0^\infty u^{-\frac{3}{2}}e^{-\frac{t^2}{4u}}W_u^{\mathcal{L}}(f)du,\;\; t>0,\;\; f \in L^p(\mathbb R^n),
$$
for every
$1 \leq p \leq \infty$.

The usual Littlewood-Paley $g$-function for the Poisson semigroup $\{P_t^{\mathcal{L}}\}_{t>0}$ is defined by
$$
g_{\mathcal{L}}(f)(x)= \left(\int_0^\infty \left|t\frac{\partial}{\partial t}P_t^{\mathcal{L}}(f)(x)\right|^2\frac{dt}{t}\right)^{\frac{1}{2}},\quad x\in \mathbb{R}^n.
$$
According to \cite[Theorem 3.2]{StTo3} and \cite[Proposition 2.3]{BFRTT} we have that, for every $1<p<\infty$,
\begin{equation}\label{acotgfn}
\frac{1}{C}\|f\|_{L^p(\mathbb{R}^n)} \leq \|g_{\mathcal{L}}(f)\|_{L^p(\mathbb{R}^n)} \leq C\|f\|_{L^p(\mathbb{R}^n)},\;\; f\in L^p(\mathbb R^n).
\end{equation}

Segovia and Wheeden (\cite{SW}) defined a fractional derivative as follows. If $F(x,t)$, $x\in \mathbb{R}^n$ and $t>0$, is a good enough function and $\sigma >0$, the $\sigma$-th $t$-derivative $\partial_t^\sigma F(x,t)$ of $F$ is given by
$$
\partial_t^\sigma F(x,t)=\frac{e^{-i\pi (m-\sigma)}}{\Gamma (m-\sigma )}\int_0^\infty \frac{\partial ^m}{\partial t^m}F(x,t+s)s^{m-\sigma -1}ds,\quad x\in \mathbb{R}^n\mbox{ and }t>0,
$$
where $m$ is the smallest integer which strictly exceeds $\sigma$. By using this fractional derivative, Segovia and Wheeden introduced a generalized Littlewood-Paley $g$-function by
$$
g_{-\Delta ,\sigma }(f)(x)=\left(\int_0^\infty |t^\sigma \partial_t^\sigma P_t^{-\Delta }(f)(x)|^2\frac{dt}{t}\right)^{\frac{1}{2}},\quad \sigma >0.
$$
Here $\{P_t^{-\Delta }\}_{t>0}$ denotes the Poisson semigroup for the classical Laplacian operator. They cha\-rac\-te\-rized Sobolev spaces in terms of the square functions $g_{-\Delta ,\sigma}$.

Inspired by \cite{SW}, in \cite{BFRTT} the generalized Littlewood-Paley function $g_{\mathcal{L},\sigma}$ associated with the Poisson semigroup for the Hermite operator is defined by
$$
g_{\mathcal{L},\sigma}(f)(x)=\left(\int_0^\infty |t^\sigma \partial_t^\sigma P_t^{\mathcal{L}}(f)(x)|^2\frac{dt}{t}\right)^{\frac{1}{2}},\quad \sigma >0.
$$
It is clear that $g_{\mathcal{L},1}=g_\mathcal{L}$. Recently, Torrea and Zhang \cite{TZ} have considered generalized Littlewood-Paley $g$-functions associated with diffusion semigroups (in the sense of Stein \cite[Chap. III]{Stein}).

By taking into account \cite[Proposition 2.3 and \S 3.2]{BFRTT} we obtain the following extension of (\ref{acotgfn}), for every $1<p<\infty$,
\begin{equation}\label{acotgfnf}
\frac{1}{C}\|f\|_{L^p(\mathbb{R}^n)}  \leq \|g_{\mathcal{L},\sigma }(f)\|_{L^p(\mathbb{R}^n)}  \leq C\|f\|_{L^p(\mathbb{R}^n)} ,\;\; f\in L^p(\mathbb R^n)\mbox{ and }\sigma >0.
\end{equation}

Our first objective is to extend the definition of the Littlewood-Paley $g_{\mathcal{L},\sigma }$-function to functions $f$ defined on $\mathbb R^n$ and taking values in a Banach space $B$ in such a suitable way that (\ref{acotgfnf}) holds for every $f \in L^p(\mathbb R^n,B)$, $1<p<\infty$.

Suppose that $B$ is a Banach space and $\sigma >0$. A first and natural definition of the Littlewood-Paley function on $L^p(\mathbb R^n,B)$, $1<p<\infty$, is the following one. If $f \in L^p(\mathbb R^n,B)$, $1<p<\infty$, we define
$$
G_{\mathcal{L},\sigma ,B}(f)(x)=\left(\int_0^\infty \left\|t^\sigma \partial_t^\sigma P_t^{\mathcal{L}}(f)(x)\right\|_B^2\frac{dt}{t}\right)^{\frac{1}{2}},\;\; x \in \mathbb R^n.
$$
This type of $g$-functions were considered, for instance, in \cite{MTX}, \cite{TZ} and \cite{Xu}. In \cite[Theorems 1 and 2]{BFRST2} it was established that $B$ is isomorphic to a Hilbert space if, and only if, for some (equivalently, for any) $1<p<\infty$, we have that
\begin{equation}\label{J1}
\frac{1}{C}\|f\|_{L^p(\mathbb{R},B)} \leq \|G_{\mathcal{L},1 , B}(f)\|_{L^p(\mathbb{R})} \leq C\|f\|_{L^p(\mathbb {R},B)},\;\; f\in L^p(\mathbb {R},B).
\end{equation}

Following the ideas developed in \cite{TZ} we can also show that $B$ is isomorphic to a Hilbert space if, and only if, for some (or equivalently, for any) $1<p<\infty$ and $\sigma >0$,
\begin{equation}\label{J2}
\frac{1}{C}\|f\|_{L^p(\mathbb R^n,B)} \leq \|G_{\mathcal{L},\sigma , B}(f)\|_{L^p(\mathbb R^n)} \leq C\|f\|_{L^p(\mathbb R^n,B)},\;\; f\in L^p(\mathbb R^n,B).
\end{equation}
The Banach spaces that are isomorphic to Hilbert spaces can also be characterized considering in (\ref{J1}) and (\ref{J2}) diffusion semigroups (\cite{MTX} and \cite{TZ}).

Motivated by the ideas in \cite{K} and \cite{KW} we are going to give a new definition of $g$-functions in the Hermite setting in $L^p(\mathbb R^n,B)$, $1<p<\infty$, for which we will prove (\ref{J2}) for all UMD Banach spaces. Note that the UMD property is stable by isomorphisms and that every Hilbert space is a UMD space but, for instance, the space $L^q(\mathbb R^n)$, $1<q<\infty$, $q \not= 2$, is a UMD space which is not isomorphic to a Hilbert space.

We consider a sequence $(\gamma_k)_{k=1}^\infty$  of independent standard Gaussian random variables on a probability space $(\Omega,  \mathcal{A}, P)$. Suppose that $H$ is a separable Hilbert space. If $T\in \mathcal{L}(H,B)$, (that is, $T$ is a bounded operator from $H$ into $B$) we say that $T$ is a $\gamma$-radonifying operator, shortly $T\in \gamma(H,B)$, when the series $\sum_{k=1}^\infty \gamma_kT(e_k)$ converges in $L^2(\Omega,B)$, where $(e_k)_{k=1}^\infty$ is an orthonormal basis in $H$, and we define
$$
\|T\|_{\gamma(H,B)}= \left(\mathbb{E}\left\|\sum_{k=1}^\infty\gamma_kT(e_k)\right\|_B^2\right)^{\frac{1}{2}}.
$$
This last quantity does not depend on the orthonormal basis $(e_k)_{k=1}^\infty$.

The space $\gamma(H,B)$ of $\gamma$-radonifying operators can be defined for every (not necessarily separable) Hilbert space but in this paper we only consider $H=L^2((0,\infty),\frac{dt}{t})$ that is a separable Hilbert space. The main properties of $\gamma$-radonifying operators can be found in \cite{Ne}.

Suppose that $H=L^2(W,\mu)$ where $(W,\mathcal{B},\mu)$ is a $\sigma$-finite measure space with a countably generated $\sigma$-algebra $\mathcal{B}$ and that $f:W \rightarrow B$ is a strongly $\mu$-measurable function such that, for every $S \in B^*$, the dual space of $B$, the function $S\circ f$ belongs to $H$.  There exists a bounded operator $T_f:H \rightarrow B$ for which
$$
\langle S, T_fh\rangle _{B^*,B}= \int_W \langle S,f(w)\rangle _{B^*,B} h(w)d\mu(w), \;\;S \in B^*\;\;\mbox{and}\;\; h \in H,
$$
where $\langle \cdot,\cdot \rangle  _{B^*,B}$ denotes the $(B^*,B)$-duality.

\noindent We say that $f \in \gamma(W,\mu,B)$ when $T_f \in \gamma(H,B)$ and then we define
$$
\|f\|_{\gamma(W,\mu,B)} = \|T_f\|_{\gamma(H,B)}.
$$
It is usual to identify $f$ with $T_f$.

Let $\alpha > -n$. The semigroup $\{W_t^{\mathcal{L}+\alpha}\}_{t> 0}$ generated by the operator $-(\mathcal{L}+\alpha)$ in $L^2(\mathbb R^n)$ can be written as follows
$$
W_t^{\mathcal{L}+\alpha}(f) = \sum_{k\in \mathbb N^n}e^{-t(2|k|+n+\alpha)}\langle f,h_k,\rangle h_k,\;\; f \in L^2(\mathbb R^n)\;\;\mbox{and}\;\;t>0.
$$
From (\ref{Hersemi}) we deduce that, for every $f\in L^2(\mathbb R^n)$,
$$
W_t^{\mathcal{L}+\alpha}(f)(x)=\int_{\mathbb R^n}W_t^\mathcal{L}(x,y)e^{-\alpha t}f(y) dy,\;\; x \in \mathbb R^n\;\;\mbox{and}\;\; t >0.
$$
By using the subordination formula, the Poisson semigroup $\{P_t^{\mathcal{L}+\alpha}\}_{t>0}$ associated with $\mathcal{L}+\alpha$ (generated by $-\sqrt{\mathcal{L}+\alpha}$) is given by
$$
P_t^{\mathcal{L}+\alpha}(f)= \frac{t}{\sqrt{4\pi}}\int_0^\infty u^{-\frac{3}{2}}e^{-\frac{t^2}{4u}-\alpha u}W_u^{\mathcal{L}}(f)du,\;\;f\in L^2(\mathbb R^n)\;\; \mbox{and}\;\; t>0.
$$

It is clear that, for every $t>0$, $W_t^{\mathcal{L}+\alpha}$ and $P_t^{\mathcal{L}+\alpha}$ are positive operators. Then, for every $t>0$, $W_t^{\mathcal{L}+\alpha}$ and $P_t^{\mathcal{L}+\alpha}$ define bounded operators from $L^2(\mathbb R^n,B)$ into itself when the integrals are considered in a $B$-Bochner sense.

We consider the space $H=L^2((0,\infty),\frac{dt}{t})$ and $\sigma >0$. Suppose that $f\in L^2(\mathbb R^n,B)$. By \cite[Lemma 2.1(ii)]{BFRTT} we have that, for every $x \in \mathbb R^n$ and $t>0$,
\begin{equation}\label{gint}
t^\sigma \partial_t^\sigma P^{\mathcal{L}+\alpha }_t(f)(x)=e^{i\pi \sigma}\sum_{k\in \mathbb N^n}(t\sqrt{2|k|+n+\alpha })^\sigma e^{-t\sqrt{2|k|+n+\alpha }}\langle f,h_k\rangle h_k(x)
\end{equation}
and it is strongly measurable (see proof of Proposition \ref{polarization}).

According to (\cite[(2.1) and Lemma 2.1]{StTo}), for every $t>0$ and $x \in \mathbb R^n$, the series in (\ref{gint}) converges in $B$. Then, for every $S \in B^*$, it follows that, for each $x\in \mathbb{R}^n$,
$$
\langle S,t^\sigma \partial_t^\sigma P_t^{\mathcal{L}+\alpha }(f)(x)\rangle  _{B^*,B}=e^{i\pi \sigma}\sum_{k\in \mathbb N^n}(t\sqrt{2|k|+n+\alpha })^\sigma e^{-t\sqrt{2|k|+n+\alpha }}\langle \langle S, f\rangle  _{B^*,B},h_k\rangle h_k(x)
$$
and by using \cite[ Lemma 2.1(ii) and Proposition 2.1(ii)]{BFRTT}, we get
$$
\left\|\|\langle S, t^\sigma \partial_t^\sigma P^{\mathcal{L}+\alpha }_t(f)\rangle  _{B^*,B}\|_{ H}\right\|_{L^2(\mathbb R^n)}
=\frac{\Gamma (2\sigma )}{2^{2\sigma}}\|\langle S,f\rangle  _{B^*,B}\|_{L^2(\mathbb R^n)}.
$$
Hence, $\langle S, t^\sigma \partial_t^\sigma P^{\mathcal{L}+\alpha }_t(f)(x)\rangle  _{B^*,B}\in H$, a.e. $x \in \mathbb R^n$.

We consider the operator $\mathcal{G}_{\mathcal{L}+\alpha,\sigma ,B}$ defined by
$$
\mathcal{G}_{\mathcal{L}+\alpha,\sigma ,B}(f)(x,t)= t^\sigma \partial_t^\sigma P_t^{\mathcal{L}+\alpha}(f)(x),\;\;x\in \mathbb R^n \;\;\mbox{and}\;\; t>0,
$$
for every $f\in L^2(\mathbb R^n,B)$. When $\sigma =1$, instead of $\mathcal{G}_{\mathcal{L}+\alpha ,1,B}$, we simply write $\mathcal{G}_{\mathcal{L}+\alpha ,B}$ .

Our first result is the following.
\begin{teo}\label{Teo1}
Let $B$ be a UMD Banach space, $\alpha >-n$ and $\sigma >0$. Then, the operator $\mathcal{G}_{\mathcal{L}+\alpha ,\sigma ,B}$ is bounded
\begin{enumerate}
\item [$(i)$] from $L^p(\mathbb R^n,B)$ into $L^p(\mathbb{R}^n, \gamma (H,B))$, for every $1<p<\infty$,
\item [$(ii$)] from $L^1(\mathbb R^n,B)$ into $L^{1,\infty}(\mathbb R^n,\gamma (H,B))$, and
\item [$(iii)$] from $H^1(\mathbb{R}^n, B)$ into $L^1(\mathbb R^n,\gamma (H,B))$.
\end{enumerate}
Moreover, for every $1<p<\infty$,
\begin{equation}\label{fcomgfu}
\|f\|_{L^p(\mathbb R^n,B)} \sim \|\mathcal{G}_{\mathcal{L}+\alpha ,\sigma ,B}(f)\|_{L^p(\mathbb R^n,\gamma(H,B))},\;\; f \in L^p(\mathbb R^n,B).
\end{equation}
\end{teo}

Note that if $f\in L^p(\mathbb{R}^n)$, $1<p<\infty$, then
$$
||\mathcal{G}_{\mathcal{L},\sigma , \mathbb{C}}(f)||_{\gamma (H,\mathbb{C})}=g_{\mathcal{L},\sigma }(f),
$$
and (\ref{fcomgfu}) is an extension of  (\ref{acotgfnf}), because $\gamma (H,\mathbb{C})=H$.

The Hermite operator $L$ can be factorized as follows
$$
L=-\frac{1}{2}\sum_{j=1}^n\left[\left(\frac{\partial}{\partial x_j}+x_j\right)\left(\frac{\partial}{\partial x_j}-x_j\right)+\left(\frac{\partial}{\partial x_j}-x_j\right)\left(\frac{\partial}{\partial x_j}+x_j\right)\right].
$$
The creation operator $-\nabla +x$ and the annihilation operator $\nabla +x$ play an important role in the harmonic analysis for the Hermite operator.

We define, for every $j=1,...,n$, the operators
$$
T_{j,\pm}^\mathcal{L}(f)=t\left(\frac{\partial}{\partial x_j}\pm x_j\right)P_t^\mathcal{L}(f),\quad f\in L^2(\mathbb{R}^n,B).
$$
$L^p$-boundedness properties of the operators $T^\mathcal{L}_{j,\pm}$ are now established.

\begin{teo}\label{Teo2}
Let $B$ be a UMD Banach space and $j=1,...,n$. The operators $T_{j,\pm}^{\mathcal{L}}$ are bounded from $L^p(\mathbb{R}^n,B)$ into  $L^p(\mathbb{R}^n,\gamma (H,B))$, for every $1<p<\infty$. In the case of $T_{j,-}^\mathcal{L}$ we assume also that $n\geq 3$.
\end{teo}

We show a new characterization of the UMD Banach spaces.

\begin{teo}\label{Teo3}
Let $B$ be a Banach space. The following assertions are equivalent.

(i) $B$ is UMD.

(ii) For some (equivalently, for every) $1<p<\infty$ and $j=1,...,n$, there exists $C>0$ such that, for every $f\in L^p(\mathbb{R}^n)\otimes B$,
 $$
 ||f||_{L^p(\mathbb{R}^n,B)}\leq C||\mathcal{G}_{\mathcal{L}+2,B}(f)||_{L^p(\mathbb{R}^n,\gamma (H,B))},
$$
and
$$
||T_{j,+}^\mathcal{L}f||_{L^p(\mathbb{R}^n,\gamma (H,B))}\leq C||f||_{L^p(\mathbb{R}^n,B)}.
$$

(iii)  For some (equivalently, for every) $1<p<\infty$ and $j=1,...,n$, there exists $C>0$ such that, for every $f\in L^p(\mathbb{R}^n)\otimes B$,
 $$
 ||f||_{L^p(\mathbb{R}^n,B)}\leq C||\mathcal{G}_{\mathcal{L}-2,B}(f)||_{L^p(\mathbb{R}^n,\gamma (H,B))},
 $$
and
$$
||T_{j,-}^\mathcal{L}f||_{L^p(\mathbb{R}^n,\gamma (H,B))}\leq C||f||_{L^p(\mathbb{R}^n,B)},
$$
where $n\geq 3$.
\end{teo}

This paper is organized as follows. In Section 2 we prove Theorem \ref{Teo1}. Theorem \ref{Teo2} and Theorem \ref{Teo3} are shown in Section 3 and Section 4, respectively. It is remarkable that in order to prove our results we use a different procedure to the one employed in \cite{KW}. In the proof of Kaiser and Weis's results (\cite{KW}) a Mikhlin type Fourier multiplier theorem (\cite[Theorem 2.12]{KW}) plays a key role. We do  not have a vector-valued multiplier theorem for the Hermite expansions. In a first step we prove the boundedness of our operators from $L^2(\mathbb{R}^n,B)$ to $L^2(\mathbb{R}^n,\gamma (H,B))$ by making a suitable comparison with the corresponding classical and Ornstein-Uhlenbeck operators and then by employing \cite[Theorem 4.2, (a), with $p=2$]{KW}. The $L^p$-boundedness properties for every $1<p<\infty$ is showed by using vector-valued Calder\'on-Zygmund theory for singular integrals. To prove the characterization of UMD Banach spaces in Theorem \ref{Teo3} we use a Cauchy-Riemann type equations in the Hermite setting which connect $\mathcal{G}_\mathcal{L}$-operators with $T^\mathcal{L}$-operators (\cite[Section 4]{StTo}).

Throughout this paper by $C$ and $c$ we always represent positive constants that can change from a line to the other one.

\section{Proof of Theorem \ref{Teo1}}
We divide the proof of Theorem \ref{Teo1} in three steps.

\subsection{} Our first objective is to see that $\mathcal{G}_{\mathcal{L}+\alpha,\sigma ,B}$ is a bounded operator from $L^2(\mathbb R^n,B)$ into $L^2(\mathbb R^n,\gamma(H,B))$.

It is known (see \cite{AT}) that harmonic analysis for the Hermite and Ornstein-Uhlenbeck operators are closely connected. We exploit this idea in order to show our objective.

The Ornstein-Uhlenbeck operator is defined on $\mathbb R^n$ by
$$
\mathfrak{O}=-\Delta + 2x\cdot \nabla.
$$
If $k=(k_1,\ldots,k_n)\in \mathbb N^n$ and $H_k(x)=\prod_{j=1}^nH_{k_j}(x_j)$, $x=(x_1,\ldots,x_n) \in \mathbb R^n$, then
$$
\mathfrak{O}H_k=2|k|H_k.
$$

The system $\{\tilde{H_k}\}_{k\in\mathbb N^n}$, where $\tilde{H_k}=\frac{2^{-|k|/2}}{k!^{\frac{1}{2}}}H_k$, ($k!=k_1!...k_n!$), $k=(k_1,...,k_n) \in \mathbb{N}^n$, is orthonormal and complete in $L^2(\mathbb R^n,d\lambda)$ where $d\lambda=\pi^{-\frac{n}{2}}e^{-|x|^2}dx$ is the Gaussian measure on $\mathbb{R}^n$. We define an extension $\mathcal{O}$ of the Ornstein-Uhlenbeck operator as follows
$$
\mathcal{O}f=\sum_{k\in \mathbb N^n}2|k|\langle f,\tilde{H_k}\rangle_{L^2(\mathbb R^n,d\lambda)}\tilde{H_k},\;\;f \in D(\mathcal{O}),
$$
where $D(\mathcal{O})=\{f\in L^2(\mathbb R^n,d\lambda):\sum_{k\in \mathbb N^n}(|k||\langle f,\tilde{H_k}\rangle_{L^2(\mathbb R^n,d\lambda)}|)^2<\infty\}$ and $\langle\cdot,\cdot \rangle_{L^2(\mathbb R^n,d\lambda)}$ denotes the usual inner product in $L^2(\mathbb R^n, d\lambda)$.

The semigroup $\{W_t^{\mathcal{O}}\}_{t>0}$ of operators generated by the operator $-\mathcal{O}$ is defined by
$$
W_t^{\mathcal{O}}(f)= \sum_{k\in\mathbb N^n}e^{-2|k|t}\langle f,\tilde{H_k}\rangle_{L^2(\mathbb R^n,d\lambda)}\tilde{H_k},\;\;t>0\;\;\mbox{and}\;\;f\in L^2(\mathbb R^n,d\lambda).
$$
According to the Mehler's formula (\cite[(1.1.36)]{Th2}) we can write, for every $f\in L^2(\mathbb R^n,d\lambda)$
\begin{equation}\label{OUKernel}
W_t^{\mathcal{O}}(f)(x)=\int_{\mathbb R^n}W_t^{\mathcal{O}}(x,y)f(y)dy,\;\; x\in \mathbb R^n\;\;\mbox{and}\;\;t>0,
\end{equation}
being
$$
W_t^{\mathcal{O}}(x,y)=\frac{1}{\pi^{\frac{n}{2}}(1-e^{-4t})^{\frac{n}{2}}}\exp\left(-\frac{|e^{-2t}x-y|^2}{1-e^{-4t}}\right), \quad x,y \in \mathbb R^n \mbox{ and }t>0.
$$
 $\{W_t^{\mathcal{O}}\}_{t>0}$, where $W_t^{\mathcal{O}}$ is defined by (\ref{OUKernel}) for every $t>0$, is a diffusion semigroup (in the sense of Stein \cite{Stein}).

Let $\beta \geq 0$. We define, for every $t>0$,
$$
W_t^{\mathcal{O}+\beta}(f)=e^{-\beta t}W_t^{\mathcal{O}}(f),\;\;f\in L^p(\mathbb R^n, d\lambda),\;\; 1 \leq p \leq \infty.
$$
The semigroup $\{W_t^{\mathcal{O}+\beta}\}_{t>0}$ of operators is generated by $-(\mathcal{O}+\beta)$ in $L^p(\mathbb R^n,d\lambda)$, $1\leq p\leq \infty$. The Poisson semigroup $\{P_t^{\mathcal{O}+\beta}\}_{t>0}$ associated with $\mathcal{O}+\beta$ is given, by using the subordination formula, as follows
$$
P^{\mathcal{O}+\beta}_t(f)=\frac{t}{\sqrt{4\pi}}\int^\infty_0u^{-\frac{3}{2}}e^{-\frac{t^2}{4u}-\beta u}W_u^{\mathcal{O}}(f)du,\;\;f\in L^p(\mathbb R^n,d\lambda),\;t>0\mbox{ and }1\leq p \leq \infty.
$$
For every $t>0$ the operators $W^{\mathcal{O}+\beta}_t$ and $P^{\mathcal{O}+\beta}_t$ are defined on $L^p(\mathbb R^n,d\lambda ,B)$, $1\leq p \leq \infty$, in a natural way, the integrals being understood in the $B$-Bochner sense.

We consider the operator
$$
\mathcal{G}_{\mathcal{O}+\beta,\sigma ,B}(f)(x,t)=t^\sigma  \partial_t^\sigma P_t^{\mathcal{O}+\beta}(f)(x),\;\; f\in L^p(\mathbb R^n,d\lambda ,B),\;\;1<p<\infty, \;x\in \mathbb R^n \;\;\mbox{and}\;\; t>0,
$$
where $\sigma >0$. Harboure, Torrea and Viviani \cite{HTV1} and Mart{\'\i}nez, Torrea and Xu \cite{MTX} have investigated $g$-functions associated to Ornstein-Uhlenbeck operator $\mathcal{O}$ in Banach valued settings. They consider the $g$-functions defined by
$$
G_{\mathcal{O},B}^q(f)(x)=\left(\int_0^\infty \Big\|t  \frac{\partial}{\partial t}P_t^{\mathcal{O}}(f)(x) \big\|_B^q\frac{dt}{t}\right)^{\frac{1}{q}},\;\;q \in (1,\infty).
$$
Note that $G_{\mathcal{O},B}^q(f)=||\mathcal{G}_{\mathcal{O},1,B}(f)||_{L^q((0,\infty ),\frac{dt}{t},B)}$. Our study follows a different way, by using $\gamma$-radonifying norms.

In order to prove Theorem \ref{Teo1} we previously need to show the following results. Let us denote by $P_t^{-\Delta}(z)$, $z\in\mathbb{R}^n$ , $t>0$, the classical Poisson kernel given by
$$
P_t^{-\Delta }(z)=b_n\frac{t}{(t^2+|z|^2)^{\frac{n+1}{2}}},\quad z\in \mathbb{R}^n \mbox{ and }t>0,
$$
where $b_n=\pi ^{-\frac{n+1}{2}}\Gamma \Big(\frac{n+1}{2}\Big)$.
\begin{lema}\label{LemaA}
Let $\sigma >0$. Then
$$
t^\sigma \partial _t^\sigma P_t^{-\Delta}(z)=\sum_{k=0}^{\frac{m+1}{2}}\frac{c_k}{t^n}\varphi ^k\left(\frac{z}{t}\right),\quad z\in \mathbb{R}^n\mbox{ and }t>0,
$$
where $m$ is the smallest integer which strictly exceeds $\sigma$, and, for every $k\in \mathbb{N}$, $0\leq k\leq\frac{m+1}{2}$, $c_k\in \mathbb{C}$ and
$$
\varphi ^k(z)=\int_0^\infty \frac{(1+v)^{m+1-2k}v^{m-\sigma -1}}{((1+v)^2+|z|^2)^{\frac{n+2(m-k)+1}{2}}}dv,\quad z\in \mathbb{R}^n.
$$
\end{lema}
\begin{proof}
Let $m$ be the smallest integer which strictly exceeds $\sigma$. We have that
$$
\partial_t^\sigma P_t^{-\Delta}(z)=b_n\frac{e^{-i\pi (m-\sigma)}}{\Gamma (m-\sigma )}\int_0^\infty \frac{\partial ^m}{\partial t^m}\left[\frac{t+s}{((t+s)^2+|z|^2)^{\frac{n+1}{2}}}\right]s^{m-\sigma -1}ds,\quad z\in \mathbb{R}^n\mbox{ and }t>0.
$$

Suppose that $n>1$. According to \cite[(4.6)]{GLLNU} we get
\begin{eqnarray}\label{Dm}
\frac{\partial ^m}{\partial t^m}\left[\frac{t+s}{((t+s)^2+|z|^2)^{\frac{n+1}{2}}}\right]&=&-\frac{1}{n-1}\frac{\partial ^{m+1}}{\partial t^{m+1}}\left[\frac{1}{((t+s)^2+|z|^2)^{\frac{n-1}{2}}}\right]\nonumber\\
&\hspace{-6cm}=&\hspace{-3cm}\sum_{k=0}^{\frac{m+1}{2}}(-1)^{m-k}E_{m+1,k}(t+s)^{m+1-2k}\frac{(n+1)(n+3)\cdots(n+2(m-k)-1)}{2^{m+1-k}((t+s)^2+|z|^2)^{\frac{n+2(m-k)+1}{2}}},\quad z\in \mathbb{R}^n, t,s>0,
\end{eqnarray}
where $E_{m+1, k}=\frac{2^{m+1-2k}(m+1)!}{k!(m+1-2k)!}$, $0\leq k\leq \frac{m+1}{2}$.

Moreover, after the change of variables $s=tv$, for every $k\in\mathbb{N}$, $0\leq k\leq \frac{m+1}{2}$, we obtain
$$
\int_0^\infty \frac{(t+s)^{m+1-2k}s^{m-\sigma -1}}{((t+s)^2+|z|^2)^{\frac{n+2(m-k)+1}{2}}}ds=t^{-n-\sigma}\int_0^\infty \frac{(1+v)^{m+1-2k}v^{m-\sigma -1}}{((1+v)^2+(|z|/t)^2)^{\frac{n+2(m-k)+1}{2}}}dv,
$$
for each $z\in \mathbb{R}^n$ and $t>0$.

If $n=1$ by taking into account that
\begin{equation}\label{Dmbis}
\frac{\partial ^m}{\partial t^m}\left[\frac {t+s}{(t+s)^2+|z|^2}\right]=\frac{1}{2}\frac{\partial ^{m+1}}{\partial t^{m+1}}[\ln ((t+s)^2+|z|^2)],\quad z\in \mathbb{R}^n\mbox{ and }t,s>0,
\end{equation}
we can proceed as above.
\end{proof}

As usual we define the Fourier transform $\widehat{f}$ of $f\in L^1(\mathbb{R}^n)$ by
$$
\widehat{f}(x)=\int_{\mathbb{R}^n}e^{-ix\cdot y}f(y)dy,\quad x\in \mathbb{R}^n.
$$
It is wellknown that $\widehat{P_t^{-\Delta}}(x)=e^{-t|x|}$, $x\in \mathbb{R}^n$ and $t>0$. It is not hard to see that, for every $\sigma >0$, $\partial_t^\sigma e^{-t|x|}=e^{i\pi \sigma }|x|^\sigma e^{-t|x|}$, $x\in \mathbb{R}^n$ and $t>0$.

\begin{lema}\label{LemaB}
Let $\sigma >0$. Then
$$
 \widehat{\partial_t^\sigma P_t^{-\Delta}}(x)=e^{i\pi \sigma }|x|^\sigma e^{-t|x|}, \quad x\in \mathbb{R}^n \mbox{ and }t>0.
$$
\end{lema}
\begin{proof}
Let $m$ be the smallest integer which strictly exceeds $\sigma$. By (\ref{Dm}) and (\ref{Dmbis}) we have that
\begin{eqnarray}\label{N1}
\left|\frac{\partial ^m}{\partial t^m}\left[\frac{t+s}{((t+s)^2+|z|^2)^{\frac{n+1}{2}}}\right]\right|&\leq& C\sum_{k=0}^{\frac{m+1}{2}}\frac{(t+s)^{m+1-2k}}{((t+s)^2+|z|^2)^{\frac{n+2(m-k)+1}{2}}}\nonumber \\
&\leq& \frac{C}{(t+|z|)^{n+\sigma /2}(t+s)^{m-\sigma /2}},\quad z\in \mathbb{R}^n\mbox{ and }t,s>0.
\end{eqnarray}
Hence
$$
\int_{\mathbb{R}^n}\int_0^\infty \left|\frac{\partial ^m}{\partial t^m}\left[\frac{t+s}{((t+s)^2+|z|^2)^{\frac{n+1}{2}}}\right]\right|s^{m-\sigma -1}dsdz<\infty ,\quad t>0,
$$
and we can change the order of integration to get
\begin{eqnarray*}
\int_{\mathbb{R}^n}e^{-ix\cdot z}\int_0^\infty \frac{\partial ^m}{\partial t^m}\left[\frac{t+s}{((t+s)^2+|z|^2)^{\frac{n+1}{2}}}\right]s^{m-\sigma -1}dsdz&&\\
&\hspace{-16cm}=&\hspace{-8cm}\int_0^\infty s^{m-\sigma -1}\int_{\mathbb{R}^n}e^{-ix\cdot z}\frac{\partial ^m}{\partial t^m}\left[\frac{t+s}{((t+s)^2+|z|^2)^{\frac{n+1}{2}}}\right]dzds,\quad x\in \mathbb{R}^n\mbox{ and }t>0.
\end{eqnarray*}
Then, by interchanging the derivatives and the integral we finish the proof.
\end{proof}

Let $B$ be a Banach space and $\sigma >0$. We consider the operator $\mathcal{G}_{-\Delta ,\sigma , B}$ defined by
$$
\mathcal{G}_{-\Delta , \sigma , B}(f)=t^\sigma \partial_t^\sigma P_t^{-\Delta }(f),\quad f\in C_c^\infty (\mathbb{R}^n)\otimes B.
$$
\begin{propo}\label{PropoC}
Let $B$ be a UMD Banach space and $\sigma >0$. Then, $\mathcal{G}_{-\Delta ,\sigma , B}$ can be extended to $L^p(\mathbb{R}^n, B)$ as a bounded operator from $L^p(\mathbb{R}^n,B)$ into $L^p(\mathbb{R}^n,\gamma (H,B))$, for every $1<p<\infty$.
\end{propo}
\begin{proof}
Let $m$ be the smallest integer which strictly exceeds $\sigma$. Assume that $f\in C_c^\infty (\mathbb{R}^n)\otimes B$, and $1<p<\infty$. We have that
$$
\partial _t^\sigma P_t^{-\Delta }(f)(x)=\int_{\mathbb{R}^n}\partial _t^\sigma P_t^{-\Delta }(x-y)f(y)dy,\quad x\in \mathbb{R}^n,\mbox{ and }t>0.
$$
To justify this equality it is sufficient to take into account (\ref{N1}).

According to Lemma \ref{LemaA} we can write
$$
t^\sigma \partial_t^\sigma P_t^{-\Delta }(f)(x)=(\varphi _t*f)(x),\quad x\in \mathbb{R}^n\mbox{ and }t>0,
$$
being $\varphi =\sum_{k=0}^{\frac{m+1}{2}}c_k\varphi ^k$, where $c_k$ and $\varphi ^k$ are defined in Lemma \ref{LemaA}, for every $k\in \mathbb{N}$, $0\leq k\leq \frac{m+1}{2}$.  Here $\varphi _t(z)=t^{-n}\varphi (z/t)$, $z\in \mathbb{R}^n$ and $t>0$. Since
$$
|\varphi ^k(z)|\leq \frac{C}{(1+|z|)^{n+\sigma }}\int_0^\infty \frac{v^{m-\sigma -1}}{(1+v)^{n+m}}dv,\quad z\in \mathbb{R}^n\mbox{ and  }0\leq k\leq \frac{m+1}{2},
$$
we have that $\varphi \in L^2(\mathbb{R}^n)$. Moreover, by using Lemma \ref{LemaB}, if $\eta\in \mathbb{N}^n$ we get
$$
\sup_{|x|=1}\int_0^\infty t^{2|\eta|}\left|\Big(\frac{d^\eta}{du^\eta}\widehat{\varphi}\Big)(tx)\right|^2\frac{dt}{t}<\infty.
$$
Then, according to \cite[Theorem 4.2]{KW} we conclude that $\mathcal{G}_{-\Delta ,\sigma ,B}$ can be extended from $C_c^\infty (\mathbb{R}^n)\otimes B$ to $L^p(\mathbb{R}^n,B)$ as a bounded operator from $L^p(\mathbb{R}^n, B)$ into $L^p(\mathbb{R}^n,\gamma (H,B))$, for every $1<p<\infty $.
\end{proof}
\begin{lema}\label{LemaD}
Let $\sigma >0$. Then,
$$
|\partial _t^\sigma [te^{-\frac{t^2}{4u}}]|\leq Ce^{-\frac{t^2}{8u}}u^{\frac{1-\sigma }{2}},\quad t,u\in(0,\infty ).
$$
\end{lema}
\begin{proof}
 Let $m$ be the smallest integer which strictly exceeds $\sigma$. By using \cite[(4.6)]{GLLNU}, we get
\begin{eqnarray*}
\frac{\partial^m}{\partial t^m}[(t+s)e^{-\frac{(t+s)^2}{4u}}]&=&-2u\frac{\partial^{m+1}}{\partial t^{m+1}}\left[e^{-\frac{(t+s)^2}{4u}}\right]=-\frac{1}{{2^m}u^{\frac{m-1}{2}}}\frac{\partial ^{m+1}}{\partial v^{m+1}}[e^{-v^2}]_{|v=\frac{t+s}{2\sqrt{u}}}\\
&=&\frac{1}{2^mu^{\frac{m-1}{2}}}\sum_{k=0}^{\frac{m+1}{2}}(-1)^{m-k}E_{m+1,k}\left(\frac{t+s}{2\sqrt{u}}\right)^{m+1-2k}e^{-\frac{(t+s)^2}{4u}},\quad t,s,u\in (0,\infty ),
\end{eqnarray*}
where $E_{m+1, k}=\frac{2^{m+1-2k}(m+1)!}{k!(m+1-2k)!}$, $0\leq k\leq \frac{m+1}{2}$.

Hence,
$$
\left|\frac{\partial^m}{\partial t^m}[(t+s)e^{-\frac{(t+s)^2}{4u}}]\right|\leq C\frac{e^{-\frac{(t+s)^2}{8u}}}{u^{\frac{m-1}{2}}}.
$$

Then,
\begin{eqnarray*}
|\partial_t^\sigma [te^{-\frac{t^2}{4u}}]|&\leq&Cu^{\frac{1-m}{2}}\int_0^\infty e^{-\frac{(t+s)^2}{8u}}s^{m-\sigma -1}ds\\
&\leq&Ce^{-\frac{t^2}{8u}}u^{\frac{1-m}{2}}\int_0^\infty e^{-\frac{s^2}{8u}}s^{m-\sigma -1}ds\\
&\leq&Ce^{-\frac{t^2}{8u}}u^{\frac{1-\sigma }{2}},\quad t,u\in (0,\infty ).
\end{eqnarray*}
\end{proof}

We will actually use the following result only when $p=2$, but we do not need to make additional  efforts to prove it for every $1<p<\infty$.

\begin{propo}\label{acotZ}
Let $B$ be a UMD space, $1<p<\infty$, $\sigma >0$ and $\beta \geq 0$. Then, the operator $\mathcal{G}_{\mathcal{O}+\beta,\sigma ,B}$ is bounded from $L^p(\mathbb R^n,d\lambda,B)$ into $L^p(\mathbb R^n,d\lambda,\gamma(H,B))$.
\end{propo}
\begin{proof}
Our procedure is inspired by the  ideas developed in \cite{HTV1}. We consider the sets
$$
\mathcal{N}=\left\{(x,y) \in \mathbb R^n \times \mathbb R^n:|x-y|<\frac{n(n+3)}{1+|x|+|y|}\right\}
$$
and
$$
\tilde{\mathcal{N}}=\left\{(x,y) \in \mathbb R^n \times \mathbb R^n:|x-y|<\frac{2n(n+3)}{1+|x|+|y|}\right\}.
$$

We choose a function $\varphi \in C^\infty(\mathbb R^n\times \mathbb R^n)$ such that
$$
\varphi(x,y)=\left\{\begin{array}{ll}
                    1, & (x,y) \in \mathcal{N},\\
                    0, & (x,y) \not\in \tilde{\mathcal{N}},
                    \end{array}\right.
$$
and that
$$
|\nabla_x\varphi(x,y)| + |\nabla_y\varphi(x,y)| \leq \frac{C}{|x-y|},\;\; x,y \in \mathbb R^n,\;\; x \not= y.
$$
The operator $\mathcal{G}_{\mathcal{O}+\beta,\sigma ,B}$ is splitted as follows
$$
\mathcal{G}_{\mathcal{O}+\beta,\sigma ,B}=\mathcal{G}_{\mathcal{O}+\beta,\sigma ,B}^{\rm loc} +\mathcal{G}_{\mathcal{O}+\beta,\sigma ,B}^{\rm glob},
$$
where, for every $f \in L^p(\mathbb R^n,d\lambda ,B)$,
$$
\mathcal{G}_{\mathcal{O}+\beta,\sigma ,B}^{\rm loc}(f)(x,t)=\mathcal{G}_{\mathcal{O}+\beta,\sigma ,B}(\varphi(x,\cdot)f)(x,t),\;\; x \in \mathbb R^n\;\;\mbox{and}\;\;t>0.
$$

We define, for every $t>0$, the operator $P_t^{-\Delta +\beta}$ by
$$
P^{-\Delta + \beta}_t(g) = \frac{t}{\sqrt{4\pi}}\int_0^\infty u^{-\frac{3}{2}} e^{-\frac{t^2}{4u} -\beta u}W_u^{-\Delta}(g) du, \;\; g\in L^p(\mathbb R^n,B).
$$
where, for every $t>0$,
$$
W_t^{-\Delta}(g)(x)=\int_{\mathbb R^n}W_t^{-\Delta }(x-y)g(y)dy,\;\; g \in L^p(\mathbb R^n,B)\;\; \mbox{and}\;\;x\in \mathbb R^n,
$$
and
$$
W_t^{-\Delta}(z)=\frac{1}{(4\pi t)^{\frac{n}{2}}}e^{-\frac{|z|^2}{4t}},\quad z\in \mathbb{R}^n.
$$

Note that $\{W_t^{-\Delta}\}_{t>0}$ represents the usual heat semigroup. We consider the operators
$$
\mathcal{G}_{-\Delta+\beta,\sigma ,B}(g)(x,t)=t^\sigma \partial_t^\sigma P_t^{-\Delta+\beta}(g)(x),\quad x\in \mathbb R^n \;\;\mbox{and}\;\;t>0,
$$
and
$$
\mathcal{G}_{-\Delta+\beta,\sigma ,B}^{\rm loc}(g)(x,t)=\mathcal{G}_{-\Delta+\beta ,\sigma ,B}(\varphi(x,\cdot)g)(x,t),\;\;x\in \mathbb R^n\;\;\mbox{and}\;\;t>0,
$$
for every $g\in L^p(\mathbb R^n,B)$.

We now split the operator $\mathcal{G}_{\mathcal{O}+\beta,\sigma ,B}$ as follows
$$
\mathcal{G}_{\mathcal{O}+\beta,\sigma ,B}=\mathcal{G}_{\mathcal{O}+\beta,\sigma ,B}^{\rm glob}+(\mathcal{G}_{\mathcal{O}+\beta,\sigma ,B}^{\rm loc}-\mathcal{G}_{-\Delta+\beta,\sigma ,B}^{\rm loc})+(\mathcal{G}_{-\Delta+\beta,\sigma, B}^{\rm loc} - \mathcal{G}_{-\Delta,\sigma ,B}^{\rm loc}) + \mathcal{G}_{-\Delta,\sigma ,B}^{\rm loc}.
$$
We study firstly the operator $\mathcal{G}_{\mathcal{O}+\beta,\sigma ,B}^{\rm loc}-\mathcal{G}_{-\Delta+\beta,\sigma ,B}^{\rm loc}$. Let $f \in L^p(\mathbb R^n,d\lambda,B)$. We have that
\begin{multline}\label{N2}
\mathcal{G}_{\mathcal{O}+\beta,\sigma ,B}^{\rm loc}(f)(x,t) - \mathcal{G}_{-\Delta+\beta,\sigma ,B}^{\rm loc}(f)(x,t)\\
= \frac{t^\sigma }{\sqrt{4\pi}}\int_0^\infty \partial_t^\sigma [te^{-\frac{t^2}{4u}}]u^{-\frac{3}{2}}
e^{-\beta u}(W_u^{\mathcal{O}}(\varphi(x,\cdot)f)(x)-W_u^{-\Delta}(\varphi(x,\cdot)f)(x))du,\;\;\mbox{a.e.}\;\;x\in \mathbb R^n\;\;\mbox{and}\;\;t>0.
\end{multline}
Indeed, let $m$ be the smallest integer which strictly exceeds $\sigma$. Then, for every $x\in \mathbb R^n$  and $t>0$,
\begin{multline*}
\mathcal{G}_{\mathcal{O}+\beta,\sigma ,B}^{\rm loc}(f)(x,t) - \mathcal{G}_{-\Delta+\beta,\sigma ,B}^{\rm loc}(f)(x,t)= \frac{t^\sigma }{\sqrt{4\pi}}\frac{e^{-i\pi (m-\sigma )}}
{\Gamma (m-\sigma )}\\
\times \int_0^\infty \frac{\partial^m}{\partial t^m}\left[(t+s)\int_0^\infty e^{-\frac{(t+s)^2}{4u}}u^{-\frac{3}{2}}e^{-\beta u}(W_u^{\mathcal{O}}(\varphi(x,\cdot)f)(x)-W_u^{-\Delta}
(\varphi(x,\cdot)f)(x))du\right]s^{m-\sigma -1}ds.
\end{multline*}

According to Lemma \ref{LemaD}, we get, for each $x\in \mathbb{R}^n\mbox{ and }t>0$,
\begin{eqnarray}\label{N3}
\int_0^\infty \int_0^\infty \left|\frac{\partial^m}{\partial t^m}\left[(t+s)e^{-\frac{(t+s)^2}{4u}}\right]\right|u^{-\frac{3}{2}}e^{-\beta u}\|W_u^{\mathcal{O}}(\varphi(x,\cdot)f)(x)-W_u^{-\Delta}
(\varphi(x,\cdot)f)(x)\|_Bdus^{m-\sigma -1}ds&&\nonumber\\
&\hspace{-28cm}\leq &\hspace{-14cm}C\int_0^\infty \int_0^\infty e^{-\frac{(t+s)^2}{8u}}u^{-1-\frac{m}{2}}e^{-\beta u}
\|W_u^\mathcal{O}(\varphi (x,\cdot)f)(x)-W_u^{-\Delta}(\varphi(x,\cdot)f)(x)\|_Bdus^{m-\sigma -1}ds\nonumber\\
&\hspace{-28cm}\leq &\hspace{-14cm}C\left(\sup_{u>0}W_u^\mathcal{O}(||f||_B)(x)+\sup_{u>0}||W_u^{-\Delta}(\varphi (x,\cdot)f)(x)||_B\right)\int_0^\infty e^{-\beta u-\frac{t^2}{8u}}u^{-1-\frac{m}{2}}\int_0^\infty e^{-\frac{s^2}{8u}}s^{m-\sigma -1}dsdu\nonumber\\
&\hspace{-24cm}\leq &\hspace{-12cm}Ct^{-\frac{\sigma }{2}}\left(\sup_{u>0}W_u^\mathcal{O}(||f||_B)(x)+\sup_{u>0}||W_u^{-\Delta}(\varphi (x,\cdot)f)(x)||_B\right) .
\end{eqnarray}

On the other hand, according to \cite[III.3]{Stein} the maximal operator
$$
W_*^\mathcal{O}(g)=\sup_{u>0}|W_u^\mathcal{O}(g)|
$$
is bounded from $L^p(\mathbb{R}^n,d\lambda )$ into itself. Hence
$$
\sup_{u>0}|W_u^\mathcal{O}(||f||_B)(x)|<\infty ,\quad \mbox{a.e.}\;x\in \mathbb{R}^n.
$$
Moreover, by using \cite[Proposition 2.4]{HTV1} we can prove that the maximal operator
$$
W_{*,B}^{-\Delta ,{\rm loc}}(g)(x)=\sup_{u>0}||W_u^{-\Delta }(\varphi (x,\cdot)g)(x)||_B
$$
is bounded from $L^p(\mathbb{R}^n,d\lambda ,B)$ into $L^p(\mathbb{R}^n,d\lambda )$. Then, $W_{*,B}^{-\Delta ,{\rm loc}}(f)(x)<\infty$, a.e. $x\in \mathbb{R}^n$.

From (\ref{N3}) we deduce that
\begin{multline*}
\int_0^\infty \int_0^\infty \left|\frac{\partial^m}{\partial t^m}\left[(t+s)e^{-\frac{(t+s)^2}{4u}}\right]\right|u^{-\frac{3}{2}}e^{-\beta u}\\
\times \|W_u^{\mathcal{O}}(\varphi(x,\cdot)f)(x)-W_u^{-\Delta}
(\varphi(x,\cdot)f)(x)\|_Bdus^{m-\sigma -1}ds<\infty, \mbox{ a.e. }x\in \mathbb{R}^n\mbox{ and }t>0.
\end{multline*}
Hence (\ref{N2}) holds.

Moreover, we can write
\begin{multline}\label{operatorZK2}
\mathcal{G}^{\rm loc}_{\mathcal{O}+\beta,\sigma ,B}(f)(x,t) - \mathcal{G}^{\rm loc}_{-\Delta+\beta,\sigma ,B}(f)(x,t)\\
= \int_{\mathbb R^n}f(y)\varphi(x,y)H_{\beta ,\sigma }(x,y,t)dy,\;\; \mbox{a.e.}\; x\in \mathbb R^n\;\;\mbox{and}\;\; t>0,
\end{multline}
where
$$
H_{\beta ,\sigma }(x,y,t)= \frac{t^\sigma }{\sqrt{4\pi}}\int_0^\infty \partial _t^\sigma [te^{-\frac{t^2}{4u}}]u^{-\frac{3}{2}}e^{-\beta u}\left(W_u^{\mathcal{O}}(x,y)-W_u^{-\Delta}(x-y)\right)du,\;\;x,y\in \mathbb R^n\;\;\mbox{and}\;\; t>0.
$$

To show (\ref{operatorZK2}) we must justify the interchange of order of integration. Since $W^{\mathcal{O}}_u(x,y) = W_u^{\mathcal{L}}(x,y)e^{\frac{|x|^2-|y|^2}{2}+nu}$, $x,y \in \mathbb R^n$ and $u>0$, and
\begin{equation}\label{acotW}
W_u^{\mathcal{L}}(x,y) \leq C\frac{e^{-nu}e^{-c\frac{|x-y|^2}{1-e^{-4u}}}}{(1-e^{-4u})^{\frac{n}{2}}}\leq C\left\{\begin{array}{ll}
                \displaystyle u^{-\frac{n}{2}}e^{-c\frac{|x-y|^2}{u}},&x,y\in \mathbb{R}^n\mbox{ and }0<u\leq 1,\\
                &\\
                \displaystyle e^{-(nu+c|x-y|^2)},&x,y\in \mathbb{R}^n\mbox{ and }u>1,
                \end{array}
\right.
\end{equation}
by using Lemma \ref{LemaD} we have that
\begin{eqnarray*}
\int_0^\infty |\partial _t^\sigma [te^{-\frac{t^2}{4u}}]|u^{-\frac{3}{2}}e^{-\beta u}|W_u^{\mathcal{O}}(x,y)-W_u^{-\Delta}(x-y)|du&&\\
&\hspace{-14cm}\leq &\hspace{-7cm} C\left(e^{\frac{|x|^2-|y|^2}{2}}+1\right)\left(\int_0^1e^{-c\frac{t^2+|x-y|^2}{u}}u^{-\frac{2+\sigma +n}{2}}du+\int_1^\infty u^{-\frac{2+\sigma }{2}}du\right)\\
&\hspace{-14cm}\leq &\hspace{-7cm} C\left(e^{\frac{|x|^2-|y|^2}{2}}+1\right)(t^{-\sigma -n}+1),\quad x,y\in \mathbb{R}^n\mbox{ and }t>0.
\end{eqnarray*}
Hence, by using H\"older inequality we get
\begin{eqnarray*}
\lefteqn{\int_{\mathbb R^n}\|f(y)\|_B|\varphi(x,y)|\int_0^\infty  |t^\sigma \partial _t^\sigma [te^{-\frac{t^2}{4u}}]|u^{-\frac{3}{2}}e^{-\beta u}|W_u^{\mathcal{O}}(x,y)-W_u^{-\Delta}(x-y)|dudy}\\
&\leq& C(t^\sigma+t^{-n})\int_{|x-y| \leq \frac{2n(n+3)}{1+|x|+|y|}}\|f(y)\|_B \left(e^{\frac{|x|^2-|y|^2}{2}}+1\right)dy<\infty , \;\; x \in \mathbb R^n\;\; \mbox{and}\;\; t>0.
\end{eqnarray*}
Then, the equality (\ref{operatorZK2}) is established.

Next, we show that
\begin{equation}\label{N4}
||H_{\beta ,\sigma }(x,y,\cdot)||_H\leq C\left(\frac{1+|x|^{\frac{1}{2}}}{|x-y|^{n-\frac{1}{2}}}+\log \frac{1}{|x-y|}\right),\quad (x,y)\in \tilde{\mathcal{N}},\;\;x\not=y.
\end{equation}
We recall that $H=L^2((0,\infty ),\frac{dt}{t})$.

Firstly we consider $\beta =0$. We proceed by following some ideas developed in \cite[proof of Lemma 3.1]{HTV1}. Note that, by taking into account Lemma \ref{LemaD}, we have that
\begin{eqnarray}\label{NA}
\lefteqn{\int_0^\infty \partial_t^\sigma [te^{-\frac{t^2}{4u}}]u^{-\frac{3}{2}}du=\partial _t^\sigma \left[t\int_0^\infty e^{-\frac{t^2}{4u}}u^{-\frac{3}{2}}du\right]}\\
&=&\frac{e^{-i\pi (m-\sigma )}}{\Gamma (m-\sigma )}\int_0^\infty \frac{\partial ^m}{\partial t^m}\int_0^\infty
(t+s)e^{-\frac{(t+s)^2}{4u}}u^{-\frac{3}{2}}dus^{m-\sigma -1}ds\nonumber\\
&=&\frac{e^{-i\pi (m-\sigma )}}{\Gamma (m-\sigma )}\int_0^\infty s^{m-\sigma -1}
\frac{\partial ^m}{\partial t^m}\int_0^\infty e^{-\frac{1}{4v}}v^{-\frac{3}{2}}dvds=0,\quad t>0.\nonumber
\end{eqnarray}
Hence, we can write
\begin{eqnarray*}
\lefteqn{H_{0,\sigma }(x,y,t)=\frac{t^\sigma}{\sqrt{4\pi }}\int_0^\infty\partial _t^\sigma [te^{-\frac{t^2}{4u}}] u^{-\frac{3}{2}}
\left(W_u^\mathcal{O}(x,y)-\frac{e^{-|y|^2}}
{\pi ^{\frac{n}{2}}}-W_u^{-\Delta }(x-y)\right)du}\\
&=&\frac{t^\sigma}{\sqrt{4\pi }}\int_0^\infty \partial _t^\sigma [te^{-\frac{t^2}{4u}}]u^{-\frac{3}{2}}\left(W_u^\mathcal{O}(x,y)-\frac{e^{-|y|^2}}{\pi ^{\frac{n}{2}}}\chi _{(1,\infty )}(u)-W_u^{-\Delta }(x-y)\right)du\\
&-&\frac{t^\sigma}{\sqrt{4\pi }}\frac{e^{-|y|^2}}{\pi ^{\frac{n}{2}}}\int_0^1\partial _t^\sigma [te^{-\frac{t^2}{4u}}]u^{-\frac{3}{2}}du
=J_1(x,y,t)+J_2(x,y,t),\quad x,y\in \mathbb{R}^n\mbox{ and }t>0.
\end{eqnarray*}

By using Minkowski's inequality and Lemma \ref{LemaD} it follows that
\begin{eqnarray*}
\|J_1(x,y,\cdot )\|_H&\leq &C\int_0^\infty u^{-\frac{2+\sigma}{2}}|W_u^\mathcal{O}(x,y)-\frac{e^{-|y|^2}}{\pi ^{\frac{n}{2}}}\chi _{(1,\infty )}(u)-W_u^{-\Delta }(x-y)| \left(\int_0^\infty t^{2\sigma -1}e^{-\frac{t^2}{4u}}dt\right)^{\frac{1}{2}}du\\
&\leq &C\int_0^\infty \frac{1}{u}\left|W_u^\mathcal{O}(x,y)-\frac{e^{-|y|^2}}{\pi ^{\frac{n}{2}}}\chi _{(1,\infty )}(u)-W_u^{-\Delta }(x-y)\right|du,\;\,x,y\in \mathbb{R}^n,\;x\not=y.
\end{eqnarray*}
Then, by \cite[Lemma 3.4]{HTV1}, we get
$$
\|J_1(x,y,\cdot )\|_H\leq C\left(\frac{1+|x|^{\frac{1}{2}}}{|x-y|^{n-\frac{1}{2}}}+\log \frac{1}{|x-y|}\right), \quad (x,y)\in \tilde{\mathcal{N}},\;x\not=y.
$$
Moreover, Lemma \ref{LemaD} leads to
\begin{eqnarray*}
\|J_2(x,y,\cdot )\|_H&\leq &Ce^{-|y|^2}\left(\int_0^\infty t^{2\sigma -1}\left(\int_0^1\partial_t^\sigma [te^{-\frac{t^2}{4u}}]u^{-\frac{3}{2}}du\right)^2dt\right)^{\frac{1}{2}}\\
&=&Ce^{-|y|^2}\left(\int_0^1t^{2\sigma -1}\left(\int_1^\infty \partial_t^\sigma [te^{-\frac{t^2}{4u}}]u^{-\frac{3}{2}}du\right)^2dt\right.\\
&+&\left.\int_1^\infty t^{2\sigma -1}\left(\int_0^1\partial_t^\sigma [te^{-\frac{t^2}{4u}}]u^{-\frac{3}{2}}du\right)^2dt\right)^{\frac{1}{2}}\\
&\leq& Ce^{-|y|^2}\left(\int_0^1 t^{2\sigma -1}dt\left(\int_1^\infty u^{-\frac{2+\sigma }{2}}du\right)^2+\int_1^\infty t^{2\sigma -1}\left(\int_0^1u^{-\frac{2+\sigma }{2}}\left(\frac{u}{t^2}\right)^{\frac{1+\sigma }{2}}du\right)^2dt\right)^{\frac{1}{2}}\\
&\leq &Ce^{-|y|^2}\leq C\frac{1+|x|^{\frac{1}{2}}}{|x-y|^{n-\frac{1}{2}}},\quad x,y\in \mathbb{R}^n,\;x\not=y.
\end{eqnarray*}
Putting together the above estimates we conclude that (\ref{N4}) holds for $\beta =0$.

Suppose now that $\beta >0$. By using again Minkowski's inequality and Lemma \ref{LemaD} we deduce that
\begin{equation}\label{Hbeta}
||H_{\beta ,\sigma }(x,y,\cdot)||_H\leq C\int_0^\infty \frac{1}{u}|W_u^\mathcal{O}(x,y)-W_u^{-\Delta }(x-y)|e^{-\beta u}du,\quad x,y\in \mathbb{R}^n,x\not=y.\end{equation}

For every $x \in  \mathbb R^n$ we define $\rho(x)=\min \{1,\frac{1}{|x|^2}\}$. We split the integral in (\ref{Hbeta}) as follows
\begin{eqnarray*}
\int_0^\infty \frac{1}{u}|W_u^{\mathcal{O}}(x,y)-W_u^{-\Delta}(x-y)|e^{-\beta u}du &=& \int_0^{\rho(x)} \frac{1}{u}|W_u^{\mathcal{O}}(x,y)-W_u^{-\Delta}(x-y)|e^{-\beta u}du \\
&&+ \int_{\rho(x)}^\infty W_u^{\mathcal{O}}(x,y) e^{-\beta u}\frac{du}{u} + \int_{\rho(x)}^\infty W_u^{-\Delta}(x-y)\frac{du}{u} \\
&=& \sum_{j=1}^3 I_j(x,y),\;\;x,y \in \mathbb R^n.
\end{eqnarray*}
As in \cite[p. 18]{HTV1} we get, for $j=1$ and $j=3$,
$$
I_j(x,y) \leq C\frac{1+|x|^{\frac{1}{2}}}{|x-y|^{n-\frac{1}{2}}}, \;\;(x,y) \in \tilde{\mathcal{N}}.
$$
Moreover, since $W_u^\mathcal{O}(x,y) \leq C \max\{u^{-\frac{n}{2}},1\}$, $x,y \in \mathbb R^n$ and $u>0$, it follows that
$$
I_2(x,y) \leq C\int_{\rho(x)}^\infty \frac{du}{u^{\frac{n}{2}+1}}  \leq C \max \{1,|x|^n\} \leq C \frac{1+|x|^{\frac{1}{2}}}{|x-y|^{n-\frac{1}{2}}}, \;\; (x,y) \in \tilde{\mathcal{N}}.
$$
Hence, we obtain
\begin{equation}\label{acotHbeta1}
\|H_{\beta , \sigma}(x,y,\cdot)\|_H \leq C\frac{1+|x|^{\frac{1}{2}}}{|x-y|^{n-\frac{1}{2}}}, \;\; (x,y) \in \tilde{\mathcal{N}}.
\end{equation}
From (\ref{N4}) we deduce that, for every $\beta \geq 0$,
$$
\sup_{x\in \mathbb R^n}\int_{\mathbb R^n}|\varphi(x,y)|\|H_{\beta ,\sigma}(x,y,\cdot)\|_Hdy < \infty,
$$
and
$$
\sup_{y \in \mathbb R^n}\int_{\mathbb R^n}|\varphi(x,y)| \|H_{\beta ,\sigma}(x,y,\cdot)\|_H dx < \infty.
$$
Hence, we conclude that the operator $\mathcal{G}^{\rm loc}_{\mathcal{O}+\beta,\sigma ,B} - \mathcal{G}^{\rm loc}_{-\Delta+\beta,\sigma ,B}$ is bounded from $L^p(\mathbb R^n,d\lambda,B)$ into $L^p(\mathbb R^n,d\lambda,\gamma(H,B))$ because $\gamma(H,\mathbb C)=H$.

In the second step we study the operator $\mathcal{G}^{\rm loc}_{-\Delta+\beta,\sigma ,B} - \mathcal{G}^{\rm loc}_{-\Delta,\sigma ,B}$. Let $f\in L^p(\mathbb R^n,d\lambda,B)$. As above, we have that
$$
\mathcal{G}^{\rm loc}_{-\Delta+\beta,\sigma ,B}(f)(x,t) - \mathcal{G}^{\rm loc}_{-\Delta,\sigma ,B}(f)(x,t) = \int_{\mathbb R^n}f(y)\varphi(x,y) \mathcal{H}_{\beta ,\sigma }(x,y,t)dy,\;\;\mbox{a.e}\;\; x\in \mathbb R^n \;\;\mbox{and}\;\; t>0,
$$
where
$$
\mathcal{H}_{\beta ,\sigma }(x,y,t)= \frac{t^\sigma }{\sqrt {4\pi}}\int_0^\infty u^{-\frac{3}{2}}\partial_t^\sigma [te^{-\frac{t^2}{4u}}] (e^{-\beta u}-1)W_u^{-\Delta}(x-y)du,\;\;x,y\in \mathbb R^n\;\;\mbox{and}\;\;t>0.
$$
By using Minkowski's inequality and Lemma \ref{LemaD} we get
$$
\|\mathcal{H}_{\beta ,\sigma }(x,y,\cdot)\|_H \leq C\int_0^\infty |e^{-\beta u}-1|W_u^{-\Delta}(x-y)\frac{du}{u},\;\; x,y \in \mathbb R^n.
$$
Moreover, we have that
\begin{eqnarray*}
\int_0^1|e^{-\beta u}-1| W_u^{-\Delta}(x-y)\frac{du}{u} &\leq& C\int_0^1 \frac{e^{-\frac{|x-y|^2}{4u}}}{u^{\frac{n}{2}}}du \leq C\int_0^1\left(\frac{u}{|x-y|^2}\right)^{\frac{n-1}{2}}\frac{du}{u^{\frac{n}{2}}} \\
&\leq& \frac{C}{|x-y|^{n-1}} \leq C\frac{1+|x|^{\frac{1}{2}}}{|x-y|^{n-\frac{1}{2}}},\;\;(x,y)\in \tilde{\mathcal{N}},\;x\not=y,
\end{eqnarray*}
and
$$
\int_1^\infty|e^{-\beta u}-1| W_u^{-\Delta}(x-y)\frac{du}{u} \leq \int_1^\infty\frac{du}{u^{\frac{n}{2}+1}} \leq C \leq C\frac{1+|x|^{\frac{1}{2}}}{|x-y|^{n-\frac{1}{2}}},\;\;(x,y)\in \tilde{\mathcal{N}},\;x\not=y.
$$
Hence,
$$
\|\mathcal{H}_{\beta, \sigma }(x,y,\cdot)\|_H\leq C\frac{1+|x|^{\frac{1}{2}}}{|x-y|^{n-\frac{1}{2}}},\;\;(x,y)\in \tilde{\mathcal{N}},\;x\not=y,
$$
and as above we deduce that the operator $\mathcal{G}^{\rm loc}_{-\Delta+\beta,\sigma ,B} - \mathcal{G}^{\rm loc}_{-\Delta,\sigma ,B}$ is bounded from $L^p(\mathbb {R}^n,d\lambda,B)$ into $L^p(\mathbb R^n,d\lambda, \gamma(H,B))$.

Minkowski's inequality leads to
$$
\left\|\mathcal{G}^{\rm glob}_{\mathcal{O}+\beta,\sigma , B}(f)(x,\cdot)\right\|_{L^2((0,\infty),\frac{dt}{t},B)} \leq \int_{\mathbb R^n}|1-\varphi(x,y)|\|f(y)\|_B ||\mathcal{G}_{\mathcal{O}+\beta, \sigma ,B}(x,y,\cdot)||_Hdy,\;\;x\in \mathbb R^n,
$$
where
$$
\mathcal{G}_{\mathcal{O}+\beta, \sigma }(x,y,t)=\frac{t^\sigma }{\sqrt{4\pi }}\int_0^\infty u^{-\frac{3}{2}}\partial_t^\sigma [te^{-\frac{t^2}{4u}}]e^{-\beta u}W_u^\mathcal{O}(x,y)du,\quad x,y\in \mathbb{R}^n\mbox{ and }t>0.
$$
We are going to estimate $||\mathcal{G}_{\mathcal{O}+\beta,\sigma }(x,y,\cdot)||_H$, $(x,y)\in \mathcal{N}^c$. Inspired by \cite[p. 1007]{Son} we define
$$
\Psi (w,z)=\int_0^w\frac{\partial ^m}{\partial z^m}[ze^{-\frac{z^2}{4v}}]v^{-\frac{3}{2}}dv,\quad w\geq 0.
$$
Note that  according to (\ref{NA}), $\Psi (\infty ,z)=0$. By partial integration we can write
\begin{eqnarray*}
\lefteqn{\mathcal{G}_{\mathcal{O}+\beta, \sigma }(x,y,t)=\frac{e^{-i\pi (m-\sigma )}t^\sigma}{\sqrt{4\pi}\Gamma (m-\sigma)}\int_0^\infty u^{-\frac{3}{2}}e^{-\beta u}W_u^\mathcal{O}(x,y)\int_0^\infty \frac{\partial ^m}{\partial t^m}[(t+s)e^{-\frac{(t+s)^2}{4u}}]s^{m-\sigma -1}dsdu}\\
&=&\frac{e^{-i\pi (m-\sigma)}t^\sigma}{\sqrt{4\pi}\Gamma (m-\sigma)}\int_0^\infty s^{m-\sigma -1}\int_0^\infty u^{-\frac{3}{2}}\frac{\partial ^m}{\partial t^m}
[(t+s)e^{-\frac{(t+s)^2}
{4u}}]e^{-\beta u }W_u^\mathcal{O}(x,y)duds\\
&=&\frac{e^{-i\pi (m-\sigma)}t^\sigma}{\sqrt{4\pi}\Gamma (m-\sigma)}\int_0^\infty s^{m-\sigma -1}\int_0^\infty \frac{\partial }{\partial u}[\Psi (u,t+s)]e^ {-\beta u}W_u^\mathcal{O}(x,y)duds\\
&=&-\frac{e^{-i\pi (m-\sigma)}t^\sigma}{\sqrt{4\pi}\Gamma (m-\sigma)}\int_0^\infty s^{m-\sigma -1}\int_0^\infty \Psi (u,t+s)\frac{\partial}{\partial u}\Big[e^ {-\beta u}W_u^\mathcal{O}(x,y)\Big]duds,\;\;x,y\in \mathbb{R}^n\mbox{ and }t>0.
\end{eqnarray*}
The interchange of the order of integrals in the second equality is justified by Lemma \ref{LemaD}.

On the other hand, by using (\ref{NA}) and Lemma \ref{LemaD} we get
\begin{eqnarray*}
\lefteqn{\left|\left|t^\sigma \int_0^\infty \psi (u,t+s)s^{m-\sigma -1}ds\right|\right|_H}\\
&=&\left(\int_0^\infty \left|t^\sigma \int_0^\infty s^{m-\sigma -1}\int_0^u\frac{\partial ^m}{\partial z^m}[ze^{-\frac{z^2}{4v}}]_{|z=t+s}v^{-\frac{3}{2}}dvds\right|^2\frac{dt}{t}\right)^{\frac{1}{2}}\\
&\leq &\left(\int_0^\infty \left|t^\sigma \int_0^{\sqrt{u}} s^{m-\sigma -1}\int_u^\infty \frac{\partial ^m}{\partial z^m}[ze^{-\frac{z^2}{4v}}]_{|z=t+s}v^{-\frac{3}{2}}dvds\right|^2\frac{dt}{t}\right)^{\frac{1}{2}} \\
&+&\left(\int_0^\infty \left|t^\sigma \int_{\sqrt{u}}^\infty s^{m-\sigma -1}\int_0^u\frac{\partial ^m}{\partial z^m}[ze^{-\frac{z^2}{4v}}]_{|z=t+s}v^{-\frac{3}{2}}dvds\right|^2\frac{dt}{t}\right)^{\frac{1}{2}}\\
&\leq &C\left[\left(\int_0^\infty \left(t^\sigma \int_0^{\sqrt{u}} s^{m-\sigma -1}\int_u^\infty e^{-\frac{(t+s)^2}{8v}}v^{-\frac{2+m}{2}}dvds\right)^2\frac{dt}
{t}\right)^{\frac{1}{2}}\right.\\
&+&\left.\left(\int_0^\infty \left(t^\sigma \int_{\sqrt{u}}^\infty s^{m-\sigma -1}\int_0^ue^{-\frac{(t+s)^2}{8v}}v^{-\frac{2+m}{2}}dvds\right)^2\frac{dt}{t}\right)^{\frac{1}{2}}\right]\\
&\leq& C\left[ \int_0^{\sqrt{u}} \int_u^\infty + \int_{\sqrt{u}}^\infty \int_0^u\right]s^{m-\sigma -1}e^{-\frac{s^2}{8v}}v^{-\frac{2+m}{2}}\left(\int_0^\infty t^{2\sigma -1}e^{-\frac{t^2}
{4v}}dt\right)^{\frac{1}{2}}dvds\\
&\leq &C\left[\int_0^{\sqrt{u}}\int_u^\infty +\int_{\sqrt{u}}^\infty\int_0^u  \right] s^{m-\sigma -1}e^{-\frac{s^2}{8v}}v^{\frac{\sigma -m-2}{2}}dvds\\
&\leq &C\left[\int_0^{\sqrt{u}}\int_0^{s^2/u} +\int_{\sqrt{u}}^\infty\int_{s^2/u}^\infty  \right]\frac{1}{s} e^{-\frac{z}{8}}z^{\frac{m-\sigma}{2}-1}dzds
\\
&\leq &C\left[u^{\frac{\sigma -m}{2}}\int_0^{\sqrt{u}}s^{m-\sigma -1}ds+u^{\frac{\sigma -m}{2}+1}\int_{\sqrt{u}}^\infty s^{m-\sigma -3}ds\right]\leq C,\quad u>0.
\end{eqnarray*}

Hence, for every $\beta >0$, Minkowski's inequality allows us to write that
\begin{eqnarray*}
||\mathcal{G}_{\mathcal{O}+\beta,\sigma }(x,y,\cdot )||_H&\leq &C\int_0^\infty \left|\frac{\partial}{\partial u}[e^{-\beta u}W_u^\mathcal{O}(x,y)]\right|du\\
&\leq &C\left(\int_0^\infty e^{-\beta u}W_u^\mathcal{O}(x,y)du+\int_0^\infty e^{-\beta u}\left|\frac{\partial}{\partial u}W_u^\mathcal{O}(x,y)\right|du\right)\\
&\leq &C\left(\sup_{u>0}W_u^\mathcal{O}(x,y)+\int_0^\infty \left|\frac{\partial}{\partial u}W_u^\mathcal{O}(x,y)\right|du\right)\\
&\leq &CK(x,y), \quad (x,y)\in \mathcal{N}^c,
\end{eqnarray*}
where
$$
K(x,y)=\left\{ \begin{array}{l}
                        e^{-|y|^2},\;\;\mbox{if}\;\;x\cdot y \leq 0,\\
        \\
                        \left(\frac{|x+y|}{|x-y|}\right)^{\frac{n}{2}}\exp\left(-\frac{|y|^2-|x|^2}{2}-\frac{|x-y||x+y|}{2}\right),\;\;\mbox{if}\;\;x\cdot y >0.\end{array}
                        \right.
$$
In the last inequality we have taken into account \cite[p. 1008, line -7]{Son} and \cite[Proposition 2.1]{MPS}.

Moreover,
$$
||\mathcal{G}_{\mathcal{O},\sigma }(x,y,\cdot )||_H\leq C\int_0^\infty \left|\frac{\partial}{\partial u}W_u^\mathcal{O}(x,y)\right|du\leq CK(x,y), \quad (x,y)\in \mathcal{N}^c.
$$
Then, according to \cite[Theorem 2.3]{MPS} (see also \cite[Lemma 2.7]{HTV1}) the operator $\mathcal{G}^{\rm glob}_{\mathcal{O}+\beta,\sigma ,B}$ is a bounded operator from $L^p(\mathbb {R}^n,d\lambda,B)$ into $L^p(\mathbb {R}^n,d\lambda,\gamma(H,B))$.

By combining all the above results we conclude that $\mathcal{G}_{\mathcal{O}+\beta,\sigma ,B}$ is bounded from $L^p(\mathbb R^n,d\lambda,B)$ into $L^p(\mathbb R^n,d\lambda,\gamma(H,B))$ if and only if $\mathcal{G}_{-\Delta, \sigma,B}^{\rm loc}$ is bounded from $L^p(\mathbb R^n,d\lambda,B)$ into $L^p(\mathbb R^n,d\lambda,\gamma(H,B))$. According to \cite[Proposition 2.4]{HTV1} $\mathcal{G}_{-\Delta, \sigma,B}^{\rm loc}$ is bounded from $L^p(\mathbb R^n,d\lambda,B)$ into $L^p(\mathbb R^n,d\lambda,\gamma(H,B))$ if and only if $\mathcal{G}_{-\Delta, \sigma ,B}^{\rm loc}$ is bounded from $L^p(\mathbb R^n,B)$ into $L^p(\mathbb R^n,\gamma(H,B))$.

By Proposition \ref{PropoC}, since $B$ is UMD, $\mathcal{G}_{-\Delta, \sigma,B}$ can be extended from $C_c^\infty (\mathbb{R}^n)\otimes B$ to $L^p(\mathbb R^n,B)$ as a bounded operator from $L^p(\mathbb R^n,B)$ into $L^p(\mathbb R^n,\gamma(H,B))$. According to \cite[Proposition 2.3]{HTV1}, it follows that $\mathcal{G}_{-\Delta, \sigma,B}^{\rm loc}$ can be extended from $C_c^\infty (\mathbb{R}^n)\otimes B$ to $L^p(\mathbb R^n,B)$ as a bounded operator from $L^p(\mathbb R^n,B)$ into $L^p(\mathbb R^n,\gamma(H,B))$. Note that in order to apply \cite[Proposition 2.3]{HTV1} we need to show that
$$
\mathcal{G}_{-\Delta, \sigma,B}(f)(x,t)= \int_{\mathbb R^n}t^\sigma\partial_t^\sigma P_t^{-\Delta}(x-y) f(y) dy,\;\; \mbox{a.e.}\;\;x \notin \rm supp\;f,
$$
for every $f \in L_c^\infty(\mathbb R^n)\otimes B$, where the integral is understood in the $\gamma(H,B)$-Bochner sense and the equality is considered in $\gamma(H,B)$. We also have to show that
$$
\|t^\sigma \partial _t^\sigma P_t^{-\Delta}(x-y)\|_H\leq \frac{C}{|x-y|^n},\;\;x,y\in \mathbb R^n\;\;\mbox{and}\;\; x \not=y.
$$
Moreover, we can see that the operator $\mathcal{G}_{\mathcal{O}+\beta , \sigma,B}$ is bounded (not only can be extended) from $L^p(\mathbb R^n,d\lambda,B)$ into $L^p(\mathbb R^n,d\lambda,\gamma(H,B))$, $1<p<\infty$.

In the next proof we need to show similar properties for an operator in the Hermite setting. We prefer to write complete proofs for the properties there because those ones will be more difficult and they will show how the properties can be proved in this case.
\end{proof}

Finally according to Proposition \ref{acotZ} and \cite[Lemma 3.1 and Proposition 3.3]{AT}, we conclude that the operator $\mathcal{G}_{\mathcal{L}+\alpha,\sigma ,B}$ is bounded from $L^2(\mathbb R^n,B)$ into $L^2(\mathbb R^n,\gamma(H,B))$.

\subsection{} Now we establish properties $(i)$ , $(ii)$ and $(iii)$ of Theorem \ref{Teo1}. We are going to use Calder\'on-Zygmund theory for vector valued singular integrals (\cite{RRT}). We consider the function
$$
\mathcal{G}_{\mathcal{L}+\alpha ,\sigma }(x,y,t) = \frac{t^\sigma }{\sqrt{4\pi}}\int_0^\infty u^{-\frac{3}{2}}\partial_t^\sigma \Big[te^{-\frac{t^2}{4u}}\Big]e^{-\alpha u}W_u^\mathcal{L}(x,y) du, \;\; x,y\in \mathbb R^n\;\; \mbox{and}\;\;t>0.
$$
Our aim is to show that, for a certain $C>0$,
\begin{equation}\label{B1}
\|\mathcal{G}_{\mathcal{L}+\alpha ,\sigma }(x,y,\cdot)\|_H \leq \frac{C}{|x-y|^n},\;\;x,y\in \mathbb R^n,\;\; x \not= y,
\end{equation}
and
\begin{equation}\label{B2}
\|\nabla_x\mathcal{G}_{\mathcal{L}+\alpha ,\sigma }(x,y,\cdot)\|_H + \|\nabla_y\mathcal{G}_{\mathcal{L}+\alpha ,\sigma }(x,y,\cdot)\|_H \leq \frac{C}{|x-y|^{n+1}},\;\;x,y\in \mathbb R^n,\;\; x \not= y.
\end{equation}
To see (\ref{B1}) and (\ref{B2}) we can proceed as in \cite[p. 114 and the following ones]{StTo3} but the calculations we now present are simpler than the ones in \cite{StTo3}.

It is no hard to see that
$$
W_t^\mathcal{L}(x,y)=\frac{1}{\pi^{\frac{n}{2}}}\left(\frac{e^{-2t}}{1-e^{-4t}}\right)^{\frac{n}{2}}\exp\left(-\frac{1}{4}\left(|x-y|^2\frac{1+e^{-2t}}{1-e^{-2t}}
+|x+y|^2\frac{1-e^{-2t}}{1+e^{-2t}}\right)\right),\;\; x,y\in \mathbb R^n\;\;\mbox{and}\;\;t>0.
$$
By using Minkowski's inequality, Lemma \ref{LemaD} and the first inequality in (\ref{acotW}) we can write
\begin{eqnarray*}
\|\mathcal{G}_{\mathcal{L}+\alpha ,\sigma }(x,y,\cdot)\|_H &\leq& C\int_0^\infty\frac{e^{-\alpha u}}{u^{\frac{3}{2}}}W_u^\mathcal{L}(x,y)\left(\int_0^\infty t^{2\sigma -1}\left|\partial _t^\sigma \Big[te^{-\frac{t^2}{4u}}\Big]\right|^2dt\right)^{\frac{1}{2}}du\\
&\leq& C \int_0^\infty\frac{e^{-\alpha u}}{u}W^\mathcal{L}_u(x,y)du\leq  C\int_0^\infty \frac{1}{u^{\frac{n}{2}+1}}e^{-c\frac{|x-y|^2}{u}}du\\
&\leq & \frac{C}{|x-y|^n},\;\; x,y \in \mathbb R^n,\;\;x \not=y.
\end{eqnarray*}
In the estimation of the integral extended to the integral $[0,1]$ we have used \cite[Lemma 1.1]{StTo}. Then (\ref{B1}) is shown.

Let $j=1,\ldots,n$. We have that
$$
\frac{\partial}{\partial x_j}\mathcal{G}_{\mathcal{L}+\alpha ,\sigma }(x,y,t) = \frac{t^\sigma }{\sqrt{4\pi}}\int_0^\infty u^{-\frac{3}{2}}\partial _t^\sigma \Big[te^{-\frac{t^2}{4u}}\Big]e^{-\alpha u}\frac{\partial}{\partial x_j}W_u^\mathcal{L}(x,y)du,\;\;x,y\in \mathbb R^n\;\;\mbox{and}\;\; t>0,
$$
where
\begin{multline}\label{B3.5}
\frac{\partial}{\partial x_j}W_u^\mathcal{L}(x,y)= -\frac{1}{2}\left(\frac{e^{-2u}}{\pi(1-e^{-4u})}\right)^{\frac{n}{2}}\left((x_j-y_j)\frac{1+e^{-2u}}{1-e^{-2u}}+(x_j+y_j)\frac{1-e^{-2u}}{1+e^{-2u}}\right)\\
\times \exp\left(-\frac{1}{4}\left(|x-y|^2\frac{1+e^{-2u}}{1-e^{-2u}}+|x+y|^2\frac{1-e^{-2u}}{1+e^{-2u}}\right)\right),\;\;x,y\in \mathbb R^n\;\;\mbox{and}\;\;u>0.
\end{multline}

Then, we obtain
\begin{eqnarray*}
\left\|\frac{\partial}{\partial x_j}\mathcal{G}_{\mathcal{L}+\alpha ,\sigma }(x,y,\cdot)\right\|_H & \leq & C\int_0^\infty e^{-\alpha u}u^{-\frac{3}{2}}\left|\frac{\partial}{\partial x_j}W_u^\mathcal{L}(x,y)\right|\left(\int_0^\infty t^{2\sigma -1}\left|\partial _t^\sigma \Big[te^{-\frac{t^2}{4u}}\Big]\right|^2dt\right)^{\frac{1}{2}}du\\
&\leq & C\int_0^\infty \frac{e^{-\alpha u}}{u}\left|\frac{\partial}{\partial x_j}W_u^\mathcal{L}(x,y)\right|du\leq C\int_0^\infty \frac{1}{u^{\frac{n+3}{2}}}e^{-c\frac{|x-y|^2}{u}}du \\
&\leq &\frac{C}{|x-y|^{n+1}},\;\;x,y\in\mathbb R^n,\;\;x\not= y.
\end{eqnarray*}
Since $\mathcal{G}_{\mathcal{L}+\alpha ,\sigma}(x,y,t)=\mathcal{G}_{\mathcal{L}+\alpha ,\sigma }(y,x,t)$, $x,y\in \mathbb R^n$ and $t >0$, (\ref{B2}) is proved.

We now define, for every $x,y\in \mathbb{R}^n$, $x\not=y$, the operator $R(x,y)$ by
$$
\begin{array}{lcl}
R(x,y):\;B&\longrightarrow &\gamma (H,B)\\
\hspace{1.5cm}b&\longrightarrow&R(x,y)(b)=\mathcal{G}_{\mathcal{L}+\alpha ,\sigma }(x,y,\cdot )b.
\end{array}
$$
Note that the definition of $R(x,y)$ is consistent. Indeed, let $x,y\in \mathbb{R}^n$, $x\not=y$. According to (\ref{B1}), for every $b\in B$, $R(x,y)(b)\in L^2((0,\infty ),\frac{dt}{t},B)$ and $R(x,y)(b)$ defines in the natural way a bounded operator from $H$ into $B$. Moreover, if $(e_k)_{k=1}^\infty$ denotes an orthonormal basis in $H$ we can write
\begin{eqnarray*}
\|R(x,y)(b)\|_{\gamma (H,B)}&=&\left(\mathbb{E}\left\|\sum_{k=1}^\infty \gamma _k\int_0^\infty \mathcal{G}_{\mathcal{L}+\alpha ,\sigma}(x,y,t)e_k(t)dtb\right\|_B^2\right)^{\frac{1}{2}}\\
&=&\|b\|_B\left(\mathbb{E}\left|\sum_{k=1}^\infty \gamma _k\int_0^\infty \mathcal{G}_{\mathcal{L}+\alpha ,\sigma}(x,y,t)e_k(t)dt\right|^2\right)^{\frac{1}{2}}\\
&\leq&C\|b\|_B\|\mathcal{G}_{\mathcal{L}+\alpha ,\sigma}(x,y,\cdot )\|_H\leq C\frac{\|b\|_B}{|x-y|^n},\quad b\in B,
\end{eqnarray*}
because $\gamma(H,\mathbb{C})=H$. This inequality shows that
$$
\|R(x,y)\|_{L(B,\gamma (H,B))}\leq C\|\mathcal{G}_{\mathcal{L}+\alpha ,\sigma}(x,y,\cdot )\|_H\leq \frac{C}{|x-y|^n}.
$$
Here $L(B,\gamma (H,B))$ denotes the space of bounded operators from $B$ into $\gamma (H,B)$. Hence, (\ref{B1}) and (\ref{B2}) imply that $R(x,y)$, $x,y\in \mathbb{R}^n$, $x\not=y$, satisfies the standard $(B,\gamma (H,B))$-Calder\'on-Zygmund conditions.

We are going to see that, for every $f\in C_c^\infty (\mathbb{R}^n)\otimes B$
\begin{equation}\label{B3}
\mathcal{G}_{\mathcal{L}+\alpha,\sigma ,B}(f)(x,t)=\left(\int_{\mathbb R^n}\mathcal{G}_{\mathcal{L}+\alpha ,\sigma }(x,y,\cdot)f(y)dy\right)(t),\quad x\not \in \mbox{ supp }f,
\end{equation}
where the integral is understood in the Bochner sense on $\gamma (H,B)$.

If $f = \displaystyle\sum_{i=1}^Na_i\phi_i$, where $a_i \in B$ and $\phi_i \in C_c^\infty(\mathbb R^n)$, $i=1,\ldots,N$, with $N\in\mathbb{N}$, we have that
$$
\mathcal{G}_{\mathcal{L}+\alpha,\sigma ,B}(f)(x,t)=\sum_{i=1}^Na_i\mathcal{G}_{\mathcal{L}+\alpha,\sigma ,\mathbb C}(\phi_i)(x,t).
$$
Hence, in order to show (\ref{B3}) it is sufficient to prove that, for every $\phi\in C_c^\infty(\mathbb R^n)$,
$$
\mathcal{G}_{\mathcal{L}+\alpha,\sigma ,\mathbb C}(\phi)(x,t)=\left(\int_{\mathbb R^n}\mathcal{G}_{\mathcal{L}+\alpha ,\sigma }(x,y,\cdot)\phi(y)dy\right)(t),
$$
where the integral is understood in the Bochner sense on $\gamma(H,\mathbb C)=H$.

Let $\phi \in C_c^\infty(\mathbb R^n)$ and $\psi \in H$. We can write
\begin{eqnarray*}
\langle \psi,\int_{\mathbb R^n}\mathcal{G}_{\mathcal{L}+\alpha ,\sigma }(x,y,\cdot)\phi(y)dy\rangle _H &=&\int_{\mathbb R^n}\langle \psi,\mathcal{G}_{\mathcal{L}+\alpha ,\sigma }(x,y,\cdot)\rangle_H\phi(y)dy \\
&=& \int_{\mathbb R^n}\int_0^\infty \psi(t)\mathcal{G}_{\mathcal{L}+\alpha ,\sigma }(x,y,t)\frac{dt}{t}\phi(y)dy,\;\; x\notin \rm supp\;\phi.
\end{eqnarray*}
According to (\ref{B1}), H\"older's inequality leads to
\begin{eqnarray*}
\int_{\mathbb R^n}\int_0^\infty |\psi(t)\mathcal{G}_{\mathcal{L}+\alpha ,\sigma }(x,y,t)|\frac{dt}{t}|\phi(y)|dy &\leq& \|\psi\|_H\int_{\mathbb R^n}|\phi(y)|\|\mathcal{G}_{\mathcal{L}+\alpha ,\sigma }(x,y,\cdot)\|_Hdy\\
&\leq & C\|\psi\|_H\int_{\rm supp\;\phi}\frac{|\phi (y)|}{|x-y|^n}dy < \infty,\;\;x\notin \rm supp\;\phi.
\end{eqnarray*}
Then, we obtain
$$
\langle\psi,\int_{\mathbb R^n}\mathcal{G}_{\mathcal{L}+\alpha ,\sigma }(x,y,\cdot)\phi(y)dy\rangle_H=\int_0^\infty\psi(t)\int_{\mathbb R^n}\mathcal{G}_{\mathcal{L}+\alpha ,\sigma }(x,y,t)\phi(y)dy\frac{dt}{t},\;\; x \notin \rm supp\phi.
$$
Hence,
$$
\mathcal{G}_{\mathcal{L}+\alpha,\sigma ,\mathbb C}(\phi )(x,t)=\left(\int_{\mathbb R^n}\mathcal{G}_{\mathcal{L}+\alpha ,\sigma }(x,y,\cdot)\phi(y)dy\right)(t),\;\;x\notin \rm supp \;\phi.
$$
By using Calder\'on-Zygmund theory for vector valued singular integrals we deduce that the operator $\mathcal{G}_{\mathcal{L}+\alpha,\sigma ,B}$ can be extended from $L^2(\mathbb R^n,B)\cap L^p(\mathbb R^n,B)$ to $L^p(\mathbb R^n,B)$ as a bounded operator $\widetilde{\mathcal{G}}_{\mathcal{L}+\alpha,\sigma ,B}$ from $L^p(\mathbb R^n,B)$ into $L^p(\mathbb R^n,\gamma(H,B))$ $1<p<\infty$, from $L^1(\mathbb R^n,B)$ into $L^{1,\infty}(\mathbb R^n,\gamma(H,B))$, and from $H^1(\mathbb R^n,B)$ into $L^{1,\infty}(\mathbb R^n,\gamma(H,B))$.

We now show that, for every $f  \in L^p(\mathbb R^n,B)$, $1\leq p < \infty$,
\begin{equation}\label{B4}
\widetilde{\mathcal{G}}_{\mathcal{L}+\alpha ,\sigma ,B}(f)(x,t)=\int_{\mathbb R^n}\mathcal{G}_{\mathcal{L}+\alpha ,\sigma }(x,y,t)f(y)dy.
\end{equation}
Let $f$ be a function in $L^p(\mathbb R^n,B)$, $1\leq p<\infty$. We choose a sequence $(f_m)_{m=1}^\infty$ in $C_c^\infty(\mathbb R^n) \otimes B$ such that $f_m \rightarrow f$, as $m \rightarrow \infty$, in $L^p(\mathbb R^n,B)$. Since $\alpha +n>0$, by Lemma \ref{LemaD} and the first inequality in (\ref{acotW}), we can write
\begin{eqnarray*}
|\mathcal{G}_{\mathcal{L}+\alpha ,\sigma }(x,y,t)| &\leq& Ct^\sigma \int_0^\infty u^{-\frac{3}{2}}\left|\partial _t^\sigma \Big[te^{-\frac{t^2}{4u}}\Big]\right|\frac{e^{-(\alpha+n)u-c\frac{|x-y|^2}{u}}}{(1-e^{-4u})^{\frac{n}{2}}}du\\
&\leq &Ct^\sigma \int_0^\infty u^{-\frac{n+4}{2}}\left|\partial _t^\sigma \Big[te^{-\frac{t^2}{4u}}\Big]\right|e^{-c\frac{|x-y|^2}{u}}du\\
&\leq&Ct^\sigma \int_0^\infty u^{-\frac{n+\sigma +3}{2}}e^{-c\frac{t^2+|x-y|^2}{u}}du\leq C\frac{t^\sigma }{(t+|x-y|)^{n+\sigma +1}},\;\;x,y\in \mathbb R^n\;\;\mbox{and}\;\;t>0.
\end{eqnarray*}
Then, for every $N\in \mathbb N$ and $x\in \mathbb R^n$,
$$
\mathcal{G}_{\mathcal{L}+\alpha,\sigma ,B}(f_m)(x,\cdot)\rightarrow \mathcal{G}_{\mathcal{L}+\alpha,\sigma ,B}(f)(x,\cdot),\;\;\mbox{as}\;\;m\rightarrow \infty,
$$
in $L^2\left(\left(\frac{1}{N},\infty\right),\frac{dt}{t},B\right)$. Since $\mathcal{G}_{\mathcal{L}+\alpha,\sigma ,B}(f_m) \rightarrow \widetilde{\mathcal{G}}_{\mathcal{L}+\alpha ,\sigma ,B}(f)$, as $m\rightarrow \infty$, in $L^p(\mathbb R^n,\gamma(H,B))$, there exist an increasing sequence $(m_k)_{k=1}^\infty$ in $\mathbb N$ and a subset $\Omega \subset \mathbb R^n$ such that $|\mathbb R^n \setminus \Omega|=0$ and, for every $x\in \Omega$,
$$
\lim_{m\rightarrow \infty} \mathcal{G}_{\mathcal{L}+\alpha,\sigma ,B}(f_{m_k})(x,\cdot)  = \widetilde{\mathcal{G}}_{\mathcal{L}+\alpha,\sigma ,B}(f)(x), \;\;\mbox{in}\;\; \gamma(H,B).
$$
Since $\gamma(H,B)$ is contained in the space $L(H,B)$ of the bounded operators from $H$ into $B$, we have that, for every $x\in \Omega$,
$$
\lim_{m\rightarrow \infty} \mathcal{G}_{\mathcal{L}+\alpha,\sigma ,B}(f_{m_k})(x,\cdot)  = \widetilde{\mathcal{G}}_{\mathcal{L}+\alpha,\sigma ,B}(f)(x), \;\;\mbox{in}\;\; L(H,B).
$$
Let $x \in \Omega$. Suppose that $h\in H$ and that the support of $h$ is a compact contained in $(0,\infty)$. Then, for every $S\in B^*$ we have that
\begin{eqnarray*}
\langle S,[\widetilde{\mathcal{G}}_{\mathcal{L}+\alpha,\sigma ,B}(f)(x)](h)\rangle_{B^*, B} &=& \lim_{k\rightarrow\infty}\langle S,[\mathcal{G}_{\mathcal{L}+\alpha,\sigma ,B}(f_{m_k})(x,\cdot)](h)\rangle_{B^*,B}\\
&=&\lim_{k\rightarrow \infty}\int_0^\infty\langle S,\mathcal{G}_{\mathcal{L}+\alpha ,\sigma ,B}(f_{m_k})(x,t)\rangle_{B^*,B}h(t)\frac{dt}{t}\\
&=&\int_0^\infty\langle S,\mathcal{G}_{\mathcal{L}+\alpha ,\sigma ,B}(f)(x,t)\rangle_{B^*,B}h(t)\frac{dt}{t},
\end{eqnarray*}
and
$$
\left|\int_0^\infty\langle S,\mathcal{G}_{\mathcal{L}+\alpha ,\sigma ,B}(f)(x,t)\rangle_{B^*,B}h(t)\frac{dt}{t}\right|\leq C\|h\|_H,
$$
where $C$ does not depend on $h$. Hence, $\langle S,\mathcal{G}_{\mathcal{L}+\alpha,\sigma ,B}(f)(x,\cdot)\rangle_{B^*,B}\in H$ and, for every $h \in H$,
$$
\int_0^\infty\langle S,\mathcal{G}_{\mathcal{L}+\alpha,\sigma ,B}(f)(x,t)\rangle_{B^*,B}h(t)\frac{dt}{t}=\langle S,[\widetilde{\mathcal{G}}_{\mathcal{L}+\alpha,\sigma ,B}(f)(x)](h)\rangle_{B^*,B}.
$$
We conclude that $\widetilde{\mathcal{G}}_{\mathcal{L}+\alpha,\sigma ,B}(f)(x)=\mathcal{G}_{\mathcal{L}+\alpha,\sigma ,B}(f)(x,\cdot)$ as elements of $\gamma(H,B)$.

\subsection{} To complete the proof of Theorem \ref{Teo1} it only remains to show that there exist $C>0$ such that
 \begin{equation}\label{16.1}
\|f\|_{L^p(\mathbb{R}^n,B)}\leq C\|\mathcal{G}_{\mathcal{L}+\alpha ,\sigma ,B}(f)\|_{L^p(\mathbb{R}^n,B)},\quad f\in L^p(\mathbb{R}^n,B).
\end{equation}

In order to prove (\ref{16.1}) we need the following polarization identity.
\begin{propo}\label{polarization}
Let $B$ be a Banach space, $\alpha >-n$ and $\sigma >0$. If $1<p<\infty$ and $f_1\in L^p(\mathbb R^n,B)$ (respectively, $L^p(\mathbb R^n) \otimes B$) and $f_2\in L^{p'}(\mathbb R^n)\otimes B^*$ (respectively, $L^{p'}(\mathbb R^n,B^*)$), then
\begin{equation}\label{pol}
\int^\infty_0 \int_{\mathbb R^n}\langle \mathcal{G}_{\mathcal{L}+\alpha,\sigma ,B}(f_1)(x,t),\mathcal{G}_{\mathcal{L}+\alpha,\sigma ,B^*}(f_2)(x,t)\rangle_{B,B^*}dx\frac{dt}{t} = \frac{e^{2\sigma \pi i}\Gamma (2\sigma )}{2^{2\sigma}}\int_{\mathbb R^n}\langle f_1(x),f_2(x)\rangle_{B,B^*} dx.
\end{equation}
\end{propo}
\begin{proof}
Let firstly $f_j \in L^2(\mathbb R^n)$, $j=1,2$. We can write,
\begin{equation}\label{M1}
P_t^{\mathcal{L}+\alpha}(f_j)(x)=\sum_{k\in\mathbb N^n} e^{-t\sqrt{2|k|+n+\alpha}}\langle f_j,h_k\rangle h_k(x),\;\; j=1,2,\;\; x\in \mathbb R^n\;\mbox{and}\; t>0.
\end{equation}
Note that, according to \cite[(2.1) and Lemma 2.3]{StTo}, the last series converges uniformly in $(x,t) \in \mathbb R^n\times (a,\infty)$, for every $a>0$, and we have that, for each $x\in \mathbb R^n\;\mbox{and}\; t>0$,
$$
\mathcal{G}_{\mathcal{L}+\alpha,\sigma ,\mathbb C}(f_j)(x,t)=e^{\sigma \pi i}\sum_{k \in \mathbb N^n} \left(t \sqrt{2|k|+n+\alpha}\right)^\sigma e^{-t\sqrt{2|k|+n+\alpha}}\langle f_j,h_k\rangle h_k(x),\;\; j=1,2 .
$$
Indeed, let $j=1,2$ and assume that $m$ is the smallest integer which strictly exceeds $\sigma$. According again to \cite[(2.1) and Lemma 2.3]{StTo} we deduce that the series in (\ref{M1}) is uniformly convergent in $(x,t)\in \mathbb{R}^n\times [a,b]$, for every $0<a<b<\infty$, and

$$
\frac{\partial ^m}{\partial t^m}P_t^{\mathcal{L}+\alpha}(f_j)(x)=\sum_{k\in \mathbb{N}^n}(-\sqrt{2|k|+n+\alpha })^me^{-t\sqrt{2|k|+n+\alpha}}\langle f_j,h_k\rangle h_k(x),\quad x\in \mathbb{R}^n\mbox{ and }t>0.
$$
We have also that
\begin{multline*}
\partial_t^\sigma P_t^{\mathcal{L}+\alpha }(f_j)(x)=\frac{e^{-i\pi (m-\sigma )}}{\Gamma (m-\sigma )}\int_0^\infty \frac{\partial ^m}{\partial t^m}[P_{t+s}^{\mathcal{L}+\alpha }(f_j)(x)]s^{m-\sigma -1}ds\\
=\frac{e^{-i\pi (m-\sigma )}}{\Gamma (m-\sigma )}\sum_{k\in \mathbb{N}^n}(-\sqrt{2|k|+n+\alpha })^me^{-t\sqrt{2|k|+n+\alpha}}\langle f_j,h_k\rangle h_k(x)\\
\times \int_0^\infty e^{-s\sqrt{2|k|+n+\alpha }}s^{m-\sigma -1}ds\\
=e^{ i\pi\sigma}\sum_{k\in \mathbb{N}^n}(\sqrt{2|k|+n+\alpha })^\sigma e^{-t\sqrt{2|k|+n+\alpha}}\langle f_j,h_k\rangle h_k(x),
\quad x\in \mathbb{R}^n\mbox{ and }t>0.
\end{multline*}
By using now Plancherel equality we conclude that
\begin{equation}\label{pol1}
\int_0^\infty\int_{\mathbb R^n}\mathcal{G}_{\mathcal{L}+\alpha,\sigma ,\mathbb C}(f_1)(x,t)\mathcal{G}_{\mathcal{L}+\alpha,\sigma ,\mathbb C}(f_2)(x,t)dx\frac{dt}{t} =  \frac{e^{2\sigma \pi i}\Gamma (2\sigma )}{2^{2\sigma}}\int_{\mathbb R^n}f_1(x)f_2(x)dx.
\end{equation}
Let $1<p<\infty$. Since $L^p(\mathbb R^n)\cap L^2(\mathbb R^n)$ is dense in $L^p(\mathbb R^n)$ and, as it was shown above, the operator $\mathcal{G}_{\mathcal{L}+\alpha,\sigma ,\mathbb C}$ is bounded from $L^q(\mathbb R^n)$ into $L^q(\mathbb{R}^n,H)$, for $q=p$ and $q=p'$, the equality (\ref{pol1}) holds for every $f_1 \in L^p(\mathbb R^n)$ and $f_2 \in L^{p'}(\mathbb R^n)$.

Suppose now that $f_1 \in L^p(\mathbb R^n,B)$ and $f_2 \in  L^{p'}(\mathbb R^n)\otimes B^*$ is defined by $f_2=\sum_{j=1}^N b_jg_j$, where $b_j \in B^*$ and $g_j \in L^{p'}(\mathbb R^n)$, $j=1,\ldots,N$, with $N \in \mathbb N$. We can write, for $t>0$ and $x\in \mathbb R^n$,
\begin{eqnarray*}
\langle \mathcal{G}_{\mathcal{L} + \alpha,\sigma ,B}(f_1)(x,t),\mathcal{G}_{\mathcal{L} + \alpha,\sigma ,B^*}(f_2)(x,t)\rangle_{B,B^*} &=& \sum_{j=1}^N\langle \mathcal{G}_{\mathcal{L} + \alpha,\sigma ,B}(f_1)(x,t),b_j\rangle_{B,B^*}\mathcal{G}_{\mathcal{L}+\alpha,\sigma ,\mathbb C}(g_j)(x,t) \\
&=& \sum_{j=1}^N \mathcal{G}_{\mathcal{L} + \alpha,\sigma ,\mathbb C}\left(\langle f_1,b_j\rangle_{B,B^*}\right)(x,t)\mathcal{G}_{\mathcal{L} + \alpha,\sigma ,\mathbb C}(g_j)(x,t).
\end{eqnarray*}
Hence, since $\langle f_1,b_j\rangle_{B,B^*} \in L^p(\mathbb R^n)$, $j=1,\ldots,N$, from (\ref{pol1}) we infer that
$$
\int_0^\infty\int_{\mathbb R^n}\langle\mathcal{G}_{\mathcal{L}+\alpha,\sigma ,B}(f_1)(x,t),\mathcal{G}_{\mathcal{L}+\alpha,\sigma ,B^*}(f_2)(x,t)\rangle_{B,B^*}dx\frac{dt}{t}= \frac{e^{2\sigma \pi i}\Gamma (2\sigma )}{2^{2\sigma}}\int_{\mathbb R^n}\langle f_1(x),f_2(x)\rangle_{B,B^*}dx.
$$
In a similar way we can prove (\ref{pol}) when $f_1 \in L^p(\mathbb R^n)\otimes B$ and $f_2 \in L^{p'}(\mathbb R^n,B^*)$.

\end{proof}

We now prove (\ref{16.1}). Let $f\in L^p(\mathbb R^n,B)$, $1<p<\infty$. We have that (\cite[Lemma 2.3]{GLY})
$$
\|f\|_{L^p(\mathbb R^n,B)} = \sup_{\substack{g\in L^{p'}_c(\mathbb R^n)\otimes B^*\\\|g\|_{L^{p'}(\mathbb R^n,B^*)}\leq 1}}\left|\int_{\mathbb R^n}\langle f(x),g(x)\rangle_{B,B^*}dx \right|.
$$
The space $B^*$ is UMD because $B$ is UMD. Then, by using \cite[Proposition 2.4]{HyWe}, Proposition \ref{polarization}, and the $L^p$-boundedness properties of the operator $\mathcal{G}_{\mathcal{L}+\alpha,\sigma ,B}$ that
we have established (Subsection 2.2), we get
\begin{eqnarray*}
\left|\int_{\mathbb R^n}\langle f(x),g(x)\rangle_{B,B^*}dx\right|&=&  \frac{2^{2\sigma}}{\Gamma (2\sigma )}\left|\int_0^\infty\int_{\mathbb R^n}\langle \mathcal{G}_{\mathcal{L}+\alpha,\sigma ,B}(f)(x,t),\mathcal{G}_{\mathcal{L}+\alpha,\sigma ,B^*}(g)(x,t)\rangle_{B,B^*}dx\frac{dt}{t}\right| \\
&\leq & C\int_{\mathbb R^n}\left\|\mathcal{G}_{\mathcal{L}+\alpha,\sigma ,B}(f)(x,\cdot)\right\|_{\gamma(H,B)}
\left\|\mathcal{G}_{\mathcal{L}+\alpha,\sigma ,B^*}(g)(x,\cdot)\right\|_{\gamma(H,B^*)}dx\\
&\leq& C\left\|\mathcal{G}_{\mathcal{L}+\alpha,\sigma ,B}(f)\right\|_{L^p(\mathbb R^n,\gamma(H,B))}
\left\|\mathcal{G}_{\mathcal{L}+\alpha,\sigma ,B^*}(g)\right\|_{L^{p'}(\mathbb R^n,\gamma(H,B^*))}\\
&\leq&C\left\|\mathcal{G}_{\mathcal{L}+\alpha,\sigma B}(f)\right\|_{L^p(\mathbb R^n,\gamma(H,B))}\|g\|_{L^{p'}(\mathbb R^n,B^*)},\;\;g\in L^{p'}_c(\mathbb R^n)\otimes B^*.
\end{eqnarray*}
Hence, we conclude that
$$
\|f\|_{L^p(\mathbb R^n,B)} \leq C\left\|\mathcal{G}_{\mathcal{L}+\alpha,\sigma ,B}(f)\right\|_{L^p(\mathbb R^n,\gamma(H,B))}.
$$
Thus, Theorem \ref{Teo1} is proved.

\section{Proof of Theorem \ref{Teo2}}

Let $j=1,2,...,n$. The $j$-th Riesz transforms $R_j^\pm$ in the Hermite setting are formally defined by
$$
R_j^\pm =\left(\frac{\partial }{\partial x_j}\pm x_j\right)\mathcal{L}^{-\frac{1}{2}},
$$
where the negative square root $\mathcal{L}^{-\frac{1}{2}}$ of $\mathcal{L}$ is given by
$$
\mathcal{L}^{-\frac{1}{2}}(f)=\frac{1}{\sqrt{\pi }}\int_0^\infty W_t^\mathcal{L}(f)\frac{dt}{\sqrt{t}}.
$$
According to \cite[(3.2)]{StTo}, if $f\in {\rm span }\{h_k\}_{k\in \mathbb{N}^n}$, we have that
\begin{equation}\label{3.1}
\left(\frac{\partial}{\partial x_j}+x_j\right)\mathcal{L}^{-\frac{1}{2}}(f)=\sum_{k\in \mathbb{N}^n}\left(\frac{2k_j}{2|k|+n}\right)^{\frac{1}{2}}\langle f,h_k\rangle h_{k-e_j},
\end{equation}
and
\begin{equation}\label{3.2}
\left(\frac{\partial}{\partial x_j}-x_j\right)\mathcal{L}^{-\frac{1}{2}}(f)=-\sum_{k\in \mathbb{N}^n}\left(\frac{2k_j+2}{2|k|+n}\right)^{\frac{1}{2}}\langle f,h_k\rangle h_{k+e_j},
\end{equation}
where $e_j$ is the $j$-th coordinate vector in $\mathbb{R}^n$. (\ref{3.1}) and (\ref{3.2}) suggest to define the operators $R_j^\pm$ on $L^2(\mathbb{R}^n)$ by
$$
R_j^+(f)=\sum_{k\in \mathbb{N}^n}\left(\frac{2k_j}{2|k|+n}\right)^{\frac{1}{2}}\langle f,h_k\rangle h_{k-e_j},\quad f\in L^2(\mathbb{R}^n),
$$
and
$$
R_j^-(f)=-\sum_{k\in \mathbb{N}^n}\left(\frac{2k_j+2}{2|k|+n}\right)^{\frac{1}{2}}\langle f,h_k\rangle h_{k+e_j},\quad f\in L^2(\mathbb{R}^n).
$$
Plancherel's theorem implies that $R_j^\pm$ are bounded operators from $L^2(\mathbb{R}^n)$ into itself. Stempak and Torrea \cite[Corollary 3.4]{StTo} established that the $j$-th Riesz transforms $R_j^\pm$ can be extended from $L^2(\mathbb{R}^n)\cap L^p(\mathbb{R}^n)$ to $L^p(\mathbb{R}^n)$ as a bounded operator from $L^p(\mathbb{R}^n)$ into itself, for every $1<p<\infty$. In \cite[Corollary 3.4]{StTo} $A_p$-weighted $L^p$ spaces are even considered.

If $B$ is a Banach space, the operators $R_j^\pm$ are defined in $L^p(\mathbb{R}^n)\otimes B$, $1<p<\infty$, in the natural way. Abu-Falahah and Torrea in \cite[Theorem 2.3]{AT} showed that $R_j^\pm$ can be extended to $L^p(\mathbb{R}^n,B)$ as a bounded operator from $L^p(\mathbb{R}^n, B)$ into itself, $1<p<\infty$, provided that $B$ is a UMD Banach space.

By combining \cite[Lemmas 4.1 and 4.2]{StTo} we can write
\begin{equation}\label{3.3}
\mathcal{G}_{\mathcal{L}\pm2,\mathbb{C}}(R_j^\pm (f))=\mp t\left(\frac{\partial }{\partial x_j}\pm x_j\right)P_t^\mathcal{L}(f),
\end{equation}
for every $f\in L^p(\mathbb{R}^n)$, $1<p<\infty$. If $B$ is a Banach space and $1<p<\infty$ from (\ref{3.3}) we deduce that
\begin{equation}\label{18.1}
\mathcal{G}_{\mathcal{L}\pm 2,B}(R_j^\pm (f))=\mp t\left(\frac{\partial }{\partial x_j}\pm x_j\right)P_t^\mathcal{L}(f),\quad f\in L^p(\mathbb{R}^n)\otimes B.
\end{equation}
According to \cite[Theorem 2.3]{AT} and Theorem \ref{Teo1}, $T_{j,\pm}^\mathcal{L}$ can be extended to $L^p(\mathbb{R}^n, B)$ as a bounded operator from $L^p(\mathbb{R}^n,B)$ into $L^p(\mathbb{R}^n,\gamma (H,B))$, provided that $1<p<\infty$ and $B$ is a UMD Banach space. In the case of $T_{j,-}^\mathcal{L}$ we also need assume that $n\geq 3$.

From (\ref{B3.5}) it follows that, for every $x,y\in \mathbb{R}^n$ and $u>0$,
\begin{eqnarray*}
\left|\left(\frac{\partial }{\partial x_j}\pm x_j\right)W_u^\mathcal{L}(x,y)\right|&\leq&C\left(\frac{e^{-2u}}{1-e^{-4u}}\right)^{\frac{n}{2}}\frac{1}{\sqrt{1-e^{-2u}}}\\
&\times&\exp \left(-\frac{1}{8}\left(|x-y|^2\frac{1+e^{-2u}}{1-e^{-2u}}+|x+y|^2\frac{1-e^{-2u}}{1+e^{-2u}}\right)\right)\\
&\leq&C\frac{e^{-nu}e^{-\frac{|x-y|^2}{u}}}{(1-e^{-2u})^{\frac{n+1}{2}}}.
\end{eqnarray*}
Then, we have that
\begin{eqnarray*}
t\left|\left(\frac{\partial }{\partial x_j}\pm x_j\right)P_t^\mathcal{L}(x,y)\right|&\leq&Ct^2\int_0^\infty u^{-\frac{3}{2}}e^{-\frac{t^2}{4u}}\left|\left(\frac{\partial }{\partial x_j}\pm x_j\right)W_u^\mathcal{L}(x,y)\right|du\\
&\leq &Ct^2\int_0^\infty \frac{e^{-c\frac{t^2+|x-y|^2}{u}}}{u^{\frac{n+5}{2}}}du \leq \frac{Ct^2}{(t+|x-y|)^{n+3}}\\
&\leq &\frac{Ct}{(t+|x-y|)^{n+2}},\quad x,y\in \mathbb{R}^n,t>0.
\end{eqnarray*}
By proceeding as in the proof of the corresponding property in Subsection 2.2 we can conclude that the operators $T_{j,\pm}^\mathcal{L}$ are bounded from $L^p(\mathbb{R}^n,B)$ into $L^p(\mathbb{R}^n,\gamma (H,B))$, for every $1<p<\infty$, when $B$ is a UMD Banach space and $n\geq 3$ in the case of $T_{j,-}^\mathcal{L}$ .

Thus the proof of Theorem \ref{Teo2} is finished.

\section{Proof of Theorem \ref{Teo3}}
Theorems \ref{Teo1} and \ref{Teo2} show that $(i)\Rightarrow (ii)$ and $(i)\Rightarrow (iii)$.

Suppose now that $(ii)$ holds for some $1<p<\infty$ and $j=1,2,...,n$. Then, by \cite[Corollary 3.4]{StTo}, (\ref{18.1}) and Theorem \ref{Teo2} we have that, for every $f\in L^p(\mathbb{R}^n)\otimes B$,
$$
||R_j^+(f)||_{L^p(\mathbb{R}^n,B)}\leq C||\mathcal{G}_{\mathcal{L}+2,B}(R_j^+(f))||_{L^p(\mathbb{R}^n,\gamma (H,B))}=C||T_{j,+}^\mathcal{L}(f)||_{L^p(\mathbb{R}^n,\gamma (H,B)) }\leq C||f||_{L^p(\mathbb{R}^n,B)}.
$$
According to \cite[Theorem 2.3]{AT} we conclude that $B$ is a UMD Banach space. Thus $(ii)\Rightarrow (i)$ is shown.

$(iii)\Longrightarrow (i)$ can be proved in a similar way.


\end{document}